\documentclass[11pt,reqno]{amsart}
\usepackage{a4wide}
\usepackage[utf8]{inputenc}
\usepackage{amsmath,amssymb,amsfonts,amsthm,mathrsfs,enumerate,cite,xcolor,graphicx}
\usepackage[pdfpagelabels,linktocpage]{hyperref}
\usepackage{url} 
\usepackage{pgf} 
\theoremstyle{plain}
\newtheorem{theorem}{Theorem}[section]
\newtheorem{lemma}[theorem]{Lemma}
\newtheorem{prop}[theorem]{Proposition}

\newtheorem{corollary}[theorem]{Corollary} 
\newtheorem{remark}[theorem]{Remark}
\newtheorem{definition}[theorem]{Definition}
\theoremstyle{definition}
\theoremstyle{remark}
\numberwithin{equation}{section}

\newcommand{\di}{d} 
\newcommand{\D}{\mathrm{D}} 
\newcommand{\RR}{\mathbb{R}}
\newcommand{\NN}{\mathbb{N}}
\newcommand{\ZZ}{\mathbb{Z}}
\renewcommand{\SS}{\mathbb{S}}
\newcommand{\II}{I\hspace{-0.2em}I}
\newcommand{\llc}{\,\text{\Large{\reflectbox{$\lrcorner$}}}}

\renewcommand{\d}{\partial}
\renewcommand{\div}{\,\mathrm{div}}
\newcommand{\dint}{\,\mathrm{d}}
\newcommand{\veps}{\varepsilon}

\newcommand{\nn}{\nonumber}
\newcommand{\wto}{\rightharpoonup}

\newcommand{\sbs}{\subset}

\newcommand{\loc}{\mathrm{loc}}
\newcommand{\supp}{\mathrm{supp}}

\newcommand{\cA}{\mathcal{A}}

\newcommand{\cD}{\mathcal{D}}
\newcommand{\cV}{\mathcal{V}}

\newcommand{\Lip}{\mathrm{Lip}}

\newcommand{\TV}{{_{\mathrm{TV}}}}

\newcommand{\diam}{\mathrm{diam}}
\newcommand{\Radon}{\mathcal{M}}
\newcommand{\Prob}{\mathcal{P}}

\newcommand{\Haus}{\mathcal{H}}
\newcommand{\MM}{\mathbb{M}}
\newcommand{\curr}[1]{T_{#1}}
\newcommand{\encvol}{\,\mathrm{encvol}}
 
\newcommand{\CH}{\mathrm{CH}}
\newcommand{\Will}{\mathcal{W}}
\newcommand{\adm}{\cA}
\newcommand{\admz}{\cA^0}
\newcommand{\admzk}{\cA_k^0}
\newcommand{\admgk}{\cA_k^g}
\newcommand{\admg}{\cA^g}

\newcommand{\comp}{Q}

\let\TeXchi\chi
\newbox\chibox
\setbox0 \hbox{\mathsurround0pt $\TeXchi$}
\setbox\chibox \hbox{\raise\dp0 \box 0 }
\def\chi{\copy\chibox}

\newcommand{\ove}{\overline}

\title[Canham-Helfrich flow of varifolds]{Generalized Minimizing Movements for \\
the varifold Canham-Helfrich flow}

\date{\today}

\author[K. Brazda]
{Katharina Brazda}
\address[K. Brazda]{Faculty of Mathematics, University of Vienna, Oskar-Morgenstern-Platz 1, A-1090 Vienna, Austria
 \& Vienna School of Mathematics (VSM).}
\email{katharina.brazda@univie.ac.at}

\author[M. Kru\v{z}\'{i}k]
{Martin Kru\v{z}\'{i}k}
\address[M. Kru\v{z}\'{i}k]{Institute of Information Theory and Automation, Czech Academy of Sciences,
Pod vod\'arenskou ve\v z\'\i\ 4, CZ-182 00, Prague 8, Czechia \& 
Faculty of Civil Engineering, Czech Technical University, Th\'{a}kurova 7, CZ-166 29, Prague 6, Czechia.}
\email{kruzik@utia.cas.cz}

\author[U. Stefanelli] {Ulisse Stefanelli} 
\address[U. Stefanelli]{Faculty of Mathematics, University of Vienna, 
Oskar-Morgenstern-Platz 1, A-1090 Vienna, Austria,\,\&
Vienna Research Platform on Accelerating Photoreaction Discovery, University of Vienna, W\"ahringerstrasse 17, A-1090 Vienna, Austria,\,\&
Istituto di Matematica Applicata e Tecnologie Informatiche E. Magenes,
via Ferrata 1, I-27100 Pavia, Italy.
}
\email{ulisse.stefanelli@univie.ac.at}

\begin{document}

\begin{abstract}
  The gradient flow of the Canham-Helfrich functional is tackled via the Generalized Minimizing Movements approach. We prove the existence of solutions  in Wasserstein spaces of varifolds,  as well as upper and lower diameter bounds. In the more regular setting of multiply covered  $C^{1,1}$ surfaces, we provide a Li-Yau-type  estimate for the Canham-Helfrich energy and prove the  conservation of multiplicity along the evolution. 
\end{abstract}



\maketitle


\section{Introduction} 

Minimizers of the {\it Canham-Helfrich energy}
\begin{equation}\label{eq:ECH} 
E_\CH(M)=\int_M \Big(\frac{\beta}{2}(H-H_0)^2+\gamma K\Big)\dint\Haus^2
\end{equation}
can be seen as models for the equilibrium shapes of single-phase biological membranes
\cite{Canham:70,Helfrich:73}. 
The membrane is represented by the closed, orientable $C^2$ surface $M$ in $\RR^3$ with $H$ and $K$ respectively denoting its mean and Gauss curvature. The material properties of the membrane material are encoded in the bending rigidities $\beta>0$, $\gamma<0$ and in the spontaneous curvature $H_0\in\RR$ and minimization is usually performed under area and enclosed-volume constraint.

We are interested in studying the {\it gradient flow dynamics} associated to the Canham-Helfrich energy $E_\CH$. This can be expected to schematically describe the dissipative evolution of single-phase biomembranes in a viscous environment, possibly up to convergence to equilibrium over time. Our aim is to tackle such evolution in the weak setting of varifolds, building on the recent varifold approach to the minimization of the Canham-Helfrich energy in \cite{BLS:20}.
Compared with stronger approaches, working with varifolds allows us to overcome the bottleneck represented by possible singularity formation and to obtain a global existence theory. The price to pay for this is the weakness of the evolution notion, as we resort in considering variational solutions in {\sc De Giorgi's}  {\it Generalized Minimizing Movements} (GMMs) sense \cite{AM:95,AmGiSa:08,DeGiorgi:93}.

Before presenting our results in detail, let us sketch a brief review of the  related literature. In the stationary case, the minimization of the Canham-Helfrich energy has been investigated in a number of different settings. 
The case of axisymmetric surfaces is considered by {\sc Choksi, Veneroni, \& Morandotti} \cite{ChoksiMorandottiVeneroni:13,ChVe:13}, also in connection with the onset of two different phases on the surface. 
Minimization in the class of uniform $C^{1,1}$ surfaces for a general class of geometric functionals including $E_\CH$ is analyzed by {\sc Dalphin} \cite{Dalphin:14, Dalphin:18}. Within the so-called parametric approach, the existence of minimizers in the framework of weak immersions is proved by {\sc Mondino \& Scharrer} \cite{MoSc:20}. 
In the weak ambient approach of oriented varifolds one has to record the result by {\sc Eichmann} \cite{Eichmann:20} and that of the already mentioned \cite{BLS:20}. In the latter, the multiphase case is also considered. Eventually, minimization in the setting of generalized Gauss graphs is studied by {\sc Kubin, Lussardi, \& Morandotti}  \cite{KuLuMo:22}. A Li-Yau-type inequality for the Canham-Helfrich energy has been recently obtained by {\sc Rupp \& Scharrer} \cite{RuSc:22}. 

In the evolutionary  case, local existence results for the classical $L^2$ gradient flow of the Canham-Helfrich functional were  obtained by {\sc Kohsaka \& Nagasawa} \cite{KoNa:06} and  {\sc Nagasawa \& Yi} \cite{NaYi:12}, see also {\sc Liu}   \cite{Liu:12} for the estimate of the lifespan of a smooth solution. In the setting of vesicles  in a viscous environment, i.e., in combination with fluid dynamics  surrounding  the volume enclosed by $M$, we   mention the local well-posedness results by {\sc K\"ohne \& Lengeler} \cite{KoLe:18, Lengeler:18} and {\sc Wang, Zhang, \& Zhang} \cite{Zhang}.  Moreover, we  refer to {\sc Elliott \& Stinner} \cite{ElSt:10},  {\sc Barrett, Garcke, \& Nürnberg} \cite{BaGaNu:17,BaGaNu:18, GaNu:21}, and {\sc Elliott \& Hatcher} \cite{ElHa:21} for the finite element approximation of gradient-flow evolution of two-phase biomembranes.

The special case of the {\it Willmore} flow corresponds to the choice $H_0=0$ and has also been specifically considered. We refer to {\sc Simonett} \cite{Simonett:01} for global existence in the vicinity of spheres and {\sc Kuwert \& Sch\"atzle} for local existence in case of small initial energy \cite{KuSc:01} and a lifespan estimate in terms of local initial curvature \cite{KuSc:02}.  A proof of convergence to the sphere for initial Willmore energy smaller than  $8\pi$ can be found in  \cite{KuSc:04}, see also the review \cite{KuSc:12}. 
{\sc Kuwert \& Scheuer} \cite{KS:21} prove stability estimates 
along the flow. The parametric approach to the Willmore flow has been treated by {\sc Palmurella \& Rivi\`ere} \cite{PR:22} who establish global well-posedness for small initial energy of the so called {\it conformal Willmore flow} and convergence to the sphere. The Willmore flow in the vicinity of the torus has been considered by {\sc Dall'Acqua, M\"uller, Sch\"atzle, \& Spener} \cite{Dall'Acqua} and the flow under additional volume or  isoperimetric constraint is studied by {\sc Rupp} \cite{Rupp:20, Rupp:21}. 
Discussions on singularity formation can be found in {\sc McCoy \& Wheeler} \cite{McWh:16} and {\sc Blatt} \cite{Blatt:09,Blatt:19}, numerical approximations are found in {\sc Barrett, Garcke, \& N\"urnberg} \cite{BaGaNu:08,BaGaNu:16}.  Finally, let us  refer to  {\sc Colli \& Lauren\c{c}ot} \cite{CoLa:12}, {\sc R\"atz \& R\"oger} \cite{RaRo:21}, and {\sc Fei \& Liu} \cite{FeLi:21} for the phase-field approach to the Willmore flow under different settings.

The focus of this work is on the Generalized-Minimizing-Movements approach to the Canham-Helfrich functional \eqref{eq:ECH} in the varifold setting, where it is reformulated as $G_\CH$, see  \cite{BLS:20}.   The very weak varifold setting is instrumental for proving the existence of equilibria without a priori structural assumptions on the minimizers. In particular, it is well-suited to handle the possible onset of singularities while evolving far from stationary points. 

The GMM evolution notion rests on a limit passage in time-discrete approximations $V^n$, which are obtained from some given initial state $V^0$ and time step $\tau>0$ via the successive in $n\in\NN$ minimization of the incremental functional
$$ 
V \mapsto G_\CH(V) + \frac{1}{2\tau}W_p^2(V,V^{n-1}),
$$
see \eqref{eq:incfunc}. The GMMs $[0,\infty)\ni t \mapsto V(t)$ are then pointwise limits of subsequences of (piecewise constant in time interpolants) of time-discrete approximations $V^n$ as $\tau\to 0$, see Definition \ref{def:GMM}. Note that the incremental functional features the interplay of energy minimization and distance-control from the previous discrete state $V^{n-1}$. In particular, we consider here the Wasserstein distance $W_p$  with $p\in[1,\infty)$  between the two varifolds $V$ and $V^{n-1}$, seen as Radon measures with equal mass. This point distinguishes our analysis from previous contributions, where $L^2$ gradient flows on different parametrization settings are considered instead.

The plan of the paper is as follows.
After  providing preliminaries in Section \ref{sec:2} and  recalling some material on the stationary case in Section \ref{sec:3}, we prove global existence of GMMs for $G_\CH$ in Section \ref{sec:4}. Contrary to previous results  on Canham-Helfrich flow, no smallness condition on the energy of the initial state is required  (Theorem \ref{thm:existence}). GMMs are absolutely continuous curves with respect to the metric $W_p$ and the Canham-Helfrich energy does not increase along the evolution. This entails that GMMs satisfy lower and upper diameter bounds (Proposition \ref{l:dia}). 
Moreover, we prove that the flow can be constrained to fulfill additional properties, and still admit the existence of a GMM (Corollary \ref{cor:restricted}). This allows us to consider an enclosed-volume constraint or to reduce to reflection- or axisymmetric varifolds. In the same spirit, evolution can be constrained to multiply covered surfaces as well, as long as uniform $C^{1,1}$ regularity is enforced as in \cite{Dalphin:18}, see also \cite{LeShSi:20}. We discuss this more regular case in Section \ref{sec:regular}, where we prove that GMMs exist for any given genus and that the multiplicity of the initial surface is conserved along the flow (Theorem \ref{thm:mult}). The minimality (hence stationarity) of multiply covered spheres is analyzed  in Subsection \ref{ssec:spheres}. Moreover, we present an a-priori bound on the multiplicity in terms of the Canham-Helfrich energy (Proposition \ref{lem:kmax}). Such Li-Yau-type bound can be compared with the recent one from \cite{RuSc:22} where however no Gauss term is considered, i.e., $\gamma=0$. Eventually, we discuss in Subsection \ref{ssec:wassweaker} the possibility of weakening the metric $W_p$ and consider instead the Wasserstein distance between the surfaces, here seen as  restricted two-dimensional Hausdorff  $\mathcal H^2$ measures in space. This choice still allows for existence of GMMs of multiply covered surfaces, as long as the multiplicity is conserved. We conclude by briefly addressing more general curvature functionals than the Canham-Helfrich energy.

\section{Notation and preliminaries}\label{sec:2}

\subsection{Measures and perimeter}

Let $X$ be a locally compact and separable metric space and $N\in\NN$.  The dual space $\mathcal{M}^N(X):=C_c(X;\RR^N)^*$ of compactly supported vector functions on $X$ is the space of vector-valued {\it Radon measures} on $X$ having  $N$ components. It is  normed by the total variation  $\|\cdot\|_{\TV}(X)$ and equipped with the {\it weak-$\ast$ topology}:
$\mu_n\wto^\ast \mu$ ($n\to\infty$) if $\int_X\varphi(x)\dint\mu_n(x)\to\int_X\varphi(x)\dint\mu(x)$ for all $\varphi\in C_c(X;\RR^N)$.  
The {\it support} of $\mu\in \mathcal{M}^N(X)$, denoted by $\supp(\mu)$, is the closed set of points $x\in X$ such that $\|\mu\|_{\TV}(U)>0$ for all neighborhoods $U$ of $x$. The {\it restriction} of $\mu$ to $B\sbs X$ is the measure $\mu\llc B$ given by $(\mu\llc B)(A):=\mu(A\cap B)$ for all Borel sets $A\sbs X$. For $\mu\in\Radon^N(X)$ and a $\mu$-measurable map $f\colon X\to Y$, where $Y$ is  another  locally compact and separable metric space, the {\it pushforward measure} of $\mu$ by $f$ is defined by $(f_\sharp\mu)(A):=\mu(f^{-1}(A))$ for all Borel sets $A\sbs Y$. If $f$ is continuous and such that $f^{-1}(B)$ is compact for all compact $B\sbs Y$, then $f_\sharp\mu\in\Radon^N(Y)$. In this case, $\supp(f_\sharp\mu)=f(\supp(\mu))$ and $\int_{Y} g(y)\dint(f_\sharp \mu)(y)=\int_X (g\circ f)(x)\dint\mu(x)$ hold for all measurable $g \colon Y\to\RR^N$.
We omit $N$ from the notation if $N=1$. The scalar Radon measure $\mu\in \mathcal{M}(X)$ is called {\it positive} if $\mu(A)\ge 0$ for every Borel set  $A\subset X$. In this case, $\mu=\|\mu\|_{\TV}$. For $1\le p<\infty$,  $L^p_\mu(X;\RR^N)$ denotes the Lebesgue space of all $\mu$-measurable functions $v\colon X\to\RR^N$ such that $\|v\|_{L^p_\mu(X)}:=(\int_X |v(x)|^p\dint\mu(x))^{1/p}<\infty$. If $X\sbs\RR^\di$ for $\di\in\NN$ and $\mu$ is the $\di$-dimensional Lebesgue measure, we simply write $\dint x$ instead of $\dint\mu(x)$.

Let $E\subset\RR^\di$ be a Lebesgue measurable set. We say that $E$ is of  {\it finite perimeter} ${\rm Per}(E)$, if there exists a Radon measure $\mu^E\in\Radon^{\di}(\RR^\di)$ such that  $\int_E(\!\div\,\varphi)\dint x=\int_{\RR^d}\varphi\,\dint\mu^E$  for every $\varphi\in C^1_c(\RR^d;\RR^\di)$ with ${\rm Per}(E):=\|\mu^E\|_{\TV}(\RR^\di)<\infty$. The {\it relative perimeter} of $E$ in a Borel set $A\subset\RR^\di$ is ${\rm Per}(E;A):=\|\mu^E\|_\TV(A)$.
By definition, $\mu^E=-\D 1_E$, where $\D$ stands for the weak derivative and $1_E\colon\RR^\di\to\{0,1\}$ is the characteristic function of $E$.
With $B_r(x)$ we indicate the open ball $B_r(x):=\{y\in\RR^d:\:|x-y|<r\}$ centered at $x\in\RR^\di$ and of radius $r>0$. The set of all points  $x\in \supp(\mu^E)$ such that $\nu(x):=\lim_{r\to 0}\mu^E(B_r(x))/|\mu^E(B_r(x))|$ exists and satisfies $\nu(x)\in \SS^{\di-1}$ is called the {\it reduced boundary} of $E$, denoted by $\d^*E$. The field $\nu$ defined in this way is called the (measure-theoretic outer unit) {\it normal} to $E$.

\subsection{Varifolds and currents}

Let  $m\in\NN_0$  with $0\leq m\leq \di$ (note however that we will only need the case $d=3$ and $m=2$ later on). The {\it Grassmannian}  $G_{m,\di}$ is  the set of all  $m$-dimensional linear subspaces of $\RR^\di$.  Elements of $G_{m,\di}$ are identified with the corresponding  orthogonal projections $P\in\RR^{d\times d}$ on the $m$-dimensional subspaces.  The {\it oriented Grassmannian}, $G_{m,\di}^o$, is the set of all oriented $m$-dimensional linear subspaces in $\RR^\di$. Its elements are represented by $m$-vectors $\xi$ in $\RR^d$.  

An {\it $m$-varifold} (later {\it varifold}, for short) in $\RR^\di$ is a positive Radon measure 
$$
V\in \Radon(\RR^\di\times G_{m,\di}).
$$ 
Similarly, an {\it oriented $m$-varifold} in $\RR^\di$ is defined on the oriented Grassmannian, i.e., $V\in \Radon(\RR^\di\times G_{m,\di}^{o})$. The {\it mass} of the varifold $V$ (either oriented or not) is the positive Radon measure $\mu_V\in\Radon(\RR^\di)$ given by 
\begin{equation*}
\mu_V(\Omega):=V(\Omega\times G_{m,\di})
\end{equation*}
for all Borel sets  $\Omega\sbs\RR^\di$. 
We introduce the 
{\it  two-fold  covering map} 
$q\colon \RR^\di\times G_{m,\di}^o\to \RR^\di\times G_{m,\di}$ given by $q(x,\pm\xi)=(x,P)$, where $P$ is the projection map onto the linear subspace spanned by $\pm\xi$.
Then,  to any  oriented varifold $V\in\Radon(\RR^\di\times G_{m,\di}^o)$, we can associate a varifold $q_\sharp V\in\Radon(\RR^\di\times G_{m,\di})$ given by the push-forward of $V$ under $q$, i.e., for all $\varphi\in C_c(\RR^\di\times G_{m,\di})$,
\begin{equation*}
(q_\sharp V)(\varphi)=\int_{\RR^\di\times G_{m,\di}}\varphi(x,P)\dint (q_\sharp V)(x,P)
=\int_{\RR^\di\times G_{m,\di}^o}\varphi(q(x,\xi))\dint V(x,\xi).
\end{equation*}
In addition, $V$ gives rise to the  {\it $m$-current} $\curr{V}\in\cD_m(\RR^\di)$  defined by
\begin{equation*}
\curr{V}(\omega)
:=\int_{\RR^\di\times G_{m,\di}^o}\langle \omega(x),\xi\rangle\dint V(x,\xi)
\end{equation*}
for all $\omega\in \cD^m(\RR^\di)$, the  smooth, compactly supported $m$-forms in $\RR^\di$. Recall that a {\it general $m$-current} in $\RR^\di$, denoted by  $T\in \cD_m(\RR^\di)$, is an element of the dual space to $\cD^m(\RR^\di)$. Its {\it mass} in $\Omega\sbs\RR^\di$ is defined by the dual norm
\begin{equation*}
\MM_T(\Omega):=\sup\{T(\omega):\,\omega\in\cD^m(\RR^\di),\, \supp(\omega)\sbs \Omega ,\,\|\omega\|_\infty\leq 1\},
\end{equation*}
where $\|\omega\|_\infty=\sup_{x\in\RR^\di}|\omega(x)|$. If $\Omega=\RR^\di$ we simply write $\MM_T$.
The  {\it boundary}  $\d T$ of $T\in \cD_m(\RR^\di)$ is the $(m-1)$-current given by $\d T(\eta):=T(\!\dint \eta)$ for all  $\eta\in\cD^{m-1}(\RR^\di)$.

A set $M\sbs\RR^\di$ ($0\leq m\leq \di$) is called {\it $m$-rectifiable} (see, e.g.,~\cite{Simon:83}) if its $m$-dimensional Hausdorff measure  $\Haus^m(M)$ is finite and if, up to a $\Haus^m$-zero set $M_0$, it is contained in a countable union of  disjoint  images $M_i$ of Lipschitz maps from $\RR^m$ to $\RR^\di$, i.e.,
$$
\textstyle{M\sbs\left(\bigcup_{i=1}^\infty M_i\right)\cup M_0}
\qquad\text{with}\qquad \Haus^m(M_0)=0.
$$ 
If $M$ is $m$-rectifiable, then its cone of {\it approximate tangent vectors} \cite[p.~7]{RatajZaehle:19} 
is an $m$-dimensional plane for $\Haus^m$-almost all $x\in M$, which we will denote by $T_xM$ as in the smooth setting. An {\it orientation} $\xi^M$ of an $m$-rectifiable set $M\sbs\RR^\di$ is a measurable choice of orientation for each tangent space $T_xM$. If $T_xM$  is spanned by $(\tau_i(x))_{i=1}^m$ with $\tau_i(x)\in\RR^\di$, then, for a.e.\ $x\in M$, one has the two choices
$\pm\xi^M(x):=\pm\tau_1(x)\wedge\cdots\wedge\tau_m(x)=(T_xM)^\pm\in G_{m,\di}^o$.
If $m=\di-1$, we can identify the orientation $\xi^M(x)$ with the unit normal vector $\nu^M(x)\in \SS^{\di-1}$. In particular, in the case $m=2$ and $d=3$, which is the relevant one in our case, we will apply the covering map in the form $q\colon \RR^3\times \SS^2\to \RR^3\times G_{2,3}$, $q(x,\pm\nu)=(x,P)$, where $P\in\RR^{3\times 3}$ is the projection map onto the plane orthogonal to $\pm\nu$. 

Let $M\sbs\RR^\di$ be $m$-rectifiable with orientation $\xi^M$. A {\it rectifiable $m$-current $T\in \cD_m(\RR^\di)$  with  integer  multiplicity}  $\theta\in L^1_{\Haus^m\llc M}(\RR^\di;\ZZ)$ is defined as 
$$
T(\omega)=\int_{M}\langle \omega(x), \xi^M(x)\rangle \,\theta(x)\dint \Haus^m(x)
$$
for all $\omega\in\cD^m(\RR^\di)$. This current will be denoted by $T=T[M,\xi^M,\theta]$ to emphasize that it is determined by the triple $M,\xi^M,\theta$. Its mass corresponds to
$
\MM_T=\int_M |\theta(x)|\dint\Haus^m(x)
$.
A rectifiable current $T$ with $\MM_T<\infty$ is called an {\it integral current}, if in addition $\MM_{\d T}<\infty$. If an integral current is supported on a smooth submanifold and its boundary is supported on the boundary of this submanifold, then the current must have constant multiplicity.  This fact is known as {\it Federer's Constancy Theorem}, 
which actually also holds in the more general case of Lipschitz submanifolds 
\cite[Cor.\ 1.54]{RatajZaehle:19}.

\begin{theorem}[{{Constancy, \cite[Sec.~4.1.31]{Federer:69}}}]\label{thm:constancy} Let $M\sbs\RR^\di$ be an $m$-dimensional, connected $C^1$ submanifold of $\RR^\di$ with boundary and assume that $M$ is oriented by $\xi^M\in C(M;G_{m,\di}^o)$. If an integral current $T\in\cD_m(\RR^\di)$  satisfies 
$\supp(T)\sbs M$ and $\supp(\d T)\sbs\d M$,
then $T$ has a constant multiplicity $\theta=c\in\ZZ$, i.e., 
$T(\omega)=c\int_{M}\langle \omega, \xi^M\rangle\dint \Haus^m$
for all $\omega\in\cD^m(\RR^\di)$.
\end{theorem}

Let $M\sbs\RR^\di$ be $m$-rectifiable and $\theta\in L^1_{\Haus^m\llc M}(\RR^\di;\NN)$. An {\it integral varifold} $V=V[M,\theta]$ is a varifold defined for all $\varphi\in C_c(\RR^\di\times G_{m,\di})$ by the formula
$$ 
V(\varphi)=\int_{M}\varphi(x,T_xM)\theta(x)\dint\Haus^m(x).$$
Equivalently, we will also write 
$\dint V(x,P)=\theta(x)\dint\Haus^m(x)\otimes\delta_{T_xM}(P)$. 
The set of all integral $m$-varifolds on $\RR^\di$ is denoted by $IV_m(\RR^\di)$. 
The mass of $V=V[M,\theta]$ corresponds to
$
\mu_V(\RR^\di)
=\int_M \theta(x)\dint\Haus^m(x)
$.
If  $\xi^M$ is an orientation of $M$ and $\theta^\pm\in L^1_{\Haus^m\llc M}(\RR^\di;\NN_0)$, then an oriented varifold $V$ defined for all $\psi\in C_c(\RR^\di\times G_{m,\di}^o)$ by  the formula
$$
V(\psi)=\int_{M}\left(\theta^+(x)\psi(x,\xi^M(x))+\theta^-(x)\psi(x,-\xi^M(x))\right)\dint \Haus^m(x)
$$
is called an {\it oriented integral varifold} and is denoted by $V=V[M,\xi^M,\theta^+,\theta^-]\in IV_m^o(\RR^\di)$.  We also write
$\dint V(x,\xi)=\dint(\Haus^m\llc M)(x)\otimes(\theta^+(x)\delta_{\xi^M(x)}(\xi)+\theta^-(x)\delta_{-\xi^M(x)}(\xi))$.
The mass of $V$ corresponds to
$
\mu_V(\RR^\di) 
=\int_M (\theta^++\theta^-)(x)\dint\Haus^m(x)
$.
If $V=V[M,\xi^M,\theta^+,\theta^-]\in IV_m^o(\RR^\di)$, then 
$$
q_\sharp V=q_\sharp V[M,\theta^++\theta^-]\in IV_m(\RR^\di)
\qquad\text{and}\qquad
\curr{V}=\curr{V}[M,\xi^M,\theta^+-\theta^-]\in \cD_m(\RR^\di).
$$
We say that $V\in IV_m(\RR^\di)$ is a {\it curvature varifold with boundary} and write $V\in AV_m(\RR^\di)$, if there exists $A^V\in L^1_{\loc,V}(\RR^\di\times G_{m,\di};\RR^{\di\times \di\times \di})$ and $\d V\in\Radon^\di(\RR^\di\times G_{m,\di})$ such that 
\begin{align}
&\int_{\RR^\di\times G_{m,\di}}\sum_{j=1}^\di\left(P_{ij}\d_j\varphi(x,P)+\sum_{k=1}^\di(\d_{P_{jk}}\varphi(x,P))\,A_{ijk}^V(x,P)+A_{jij}^V(x,P)\,\varphi(x,P)\right)\dint V(x,P)\nn\\
&=-\int_{\RR^\di\times G_{m,\di}}\varphi(x,P)\dint (\d V)_i(x,P)\label{eq:defAV}
\end{align}
for all $\varphi\in C_c^1(\RR^\di\times G_{m,\di})$ and $1\leq i\leq \di$. The {\it curvature functions} $A^V_{ijk}$ ($1\leq i,j,k\leq d$) give rise to the {\it generalized second fundamental form} $\II^V\in L^1_{\loc,V}(\RR^\di\times G_{m,\di};\RR^{\di\times \di\times \di})$ via
$$
(\II^V(\cdot,P))_{ij}^k:=\sum_{l=1}^\di A_{ikl}^V(\cdot,P)P_{lj}
$$
and the vector-valued Radon measure $\d V$ is called the {\it boundary measure}. The set of {\it oriented curvature varifolds}, $AV_m^o(\RR^\di)$, is given by all $V\in IV_m^o(\RR^\di)$ such that $q_\sharp V\in
AV_m(\RR^\di)$.

\subsection{Compactness}

Let us recall the {\it compactness theorems} for oriented integral varifolds by Hutchinson \cite{Hutchinson:86} and for curvature varifolds with boundary by Mantegazza \cite{Mantegazza:96}. Let $\Omega\sbs\RR^\di$ be open. The classes of varifolds defined on $\Omega$ are denoted by $IV_m(\Omega)$, $IV_m^o(\Omega)$, and $AV_m(\Omega)$.

\begin{theorem}[Compactness, 
{\cite[Thm.\ 3.1]{Hutchinson:86}}]\label{thm:IVmcpt} The set 
$$
\left\{V\in IV^o_m(\Omega):\:\mu_{V}(\Omega)+\|\delta(q_\sharp V)\|_{\TV}(\Omega)+\mathbb{M}_{\d \curr{V}}(\Omega)<\infty \right\}
$$
is compact in the weak-$\ast$ topology of $\Radon(\Omega\times G_{m,\di}^o)$.
\end{theorem}
Here, $\delta(q_\sharp V)\colon C^1_c(\Omega;\RR^\di)\to\RR$ stands for the first variation \cite{Allard:72} of the varifold $q_\sharp V$, which can be estimated in terms of mass, curvature, and boundary \cite[Lemma
2.3]{BLS:20},
$$
\|\delta(q_\sharp V)\|_\TV(\Omega)
\leq \sqrt{\di\,\mu_{V}(\Omega)}\, \|A^{q_\sharp V}\|_{L^2_{q_\sharp V}(\Omega\times G_{m,\di})}+\|\d (q_\sharp  V)\|_\TV(\Omega\times G_{m,\di}).
$$ 
The {\it curvature varifold convergence} $V_n\wto^\ast V$ in $AV_m(\Omega)$ ($n\to\infty$) is defined as the {\it measure-function pairs convergence} $(V_n,A^{V_n})\wto^\ast (V,A^{V})$: $V_n\wto^\ast V$ in $\Radon(\Omega\times G_{m,\di})$ and  $A^{V_n}V_n\wto^\ast A^{V}V$ in $\Radon^{\di\times\di\times\di}(\Omega\times G_{m,\di})$, cf.~\cite{Hutchinson:86}. In particular, if $V_n\wto^\ast V$ in $AV_m(\Omega)$ then, by definition \eqref{eq:defAV}, $\d V_n\wto^\ast\d V$ in $\Radon^\di(\Omega\times G_{m,\di})$.

\begin{theorem}[Compactness,
{\cite[Thm.\ 6.1]{Mantegazza:96}}]\label{thm:AVmcpt} The set 
$$
\left\{V\in AV_m(\Omega):\:\mu_V(\Omega)+\|A^V\|^2_{L_V^2(\Omega\times G_{m,\di})}+\|\d V\|_{\TV}(\Omega\times G_{m,\di})<\infty\right\}
$$
is compact with respect to curvature varifold convergence. 
\end{theorem}

\subsection{Wasserstein space}\label{ssec:Wasserstein}

Let $(X, d_X)$ be a complete separable metric space and 
$$
\Prob(X):=\{\mu\in\Radon(X):\mu(X)=m_0\}
$$
indicate the set of positive Radon measures with fixed mass $m_0>0$.
For $p\in [ 1,\infty)$, let 
$
\Prob_p(X):=\{\mu\in\Prob(X):\:\int_X d_X^p(x,x_0)\,\dint\mu(x)<\infty\}
$
for some $x_0\in X$ denote its subspace of measures with {\it finite $p$-th moment} and recall that $\Prob_p(X)=\Prob(X)$ for all $p$ if $X$ is compact. The $p$-{\it Wasserstein distance} between $\mu,\tilde\mu\in\Prob_p(X)$ is defined  by 
\begin{equation*}
W_p(\mu,\tilde\mu):=\left(\min_{\lambda\in\Pi(\mu,\tilde\mu)}\int_{X\times X} d_X^p(x,\tilde x)\,\dint\lambda(x,\tilde x)\right)^{1/p},
\end{equation*} 
where 
$
\Pi(\mu,\tilde\mu):=\left\{\lambda\in\Radon(X\times X):\:(\pi_1)_\sharp\lambda=\mu,\: (\pi_2)_\sharp\lambda=\tilde\mu\right\}
$
is the set of all couplings with marginals $\mu$ and $\tilde\mu$. The metric space $(\Prob_p(X),W_p)$ is called the $p$-{\it Wasserstein space} on $(X, d_X)$.
If $\mu,\tilde\mu\in\Prob_1(X)$ and have bounded support, then by \cite[Thm.\ 8.10.45]{Bogachev2:07} and \cite[(7.1.2)]{AmGiSa:08},
$$
W_1(\mu,\tilde\mu) = \sup\left\{{\int_X f\,\dint(\mu-\tilde\mu):\:f\in\Lip_1(X)}\right\},
$$
which is known as the {\it Kantorovich-Rubinshtein distance} (or {\it modified bounded Lipschitz distance}). Here, $\Lip_1$ stands for Lipschitz continuous functions with unit Lipschitz constant. This representation renders the $W_1$ distance particularly useful for applications, see \cite{BuLeMa:17, BuLeMa:22}. 

The metric space $(\Prob_p(X),W_p)$ is complete and separable  \cite[Prop.\ 7.1.5]{AmGiSa:08}. Moreover, $W_p$ generates the weak-$\ast$ topology on $\Prob_p(X)$: If $(\mu_n)\sbs \Prob_p(X)$, $\mu\in\Prob_p(X)$, and $n\to\infty$, then 
$$
W_p(\mu_n,\mu)\to 0
\:\:\Longleftrightarrow\:\:
\mu_n\wto^\ast\mu\:\:\text{and}\:\:\text{$\int_Xf\dint\mu_n\to\int_Xf\dint\mu\:\:$ for $f\in C(X)$ with $p$-growth,}
$$
 i.e., $|f|\le c(1+d_X^p(\cdot,x_0))$ for some $x_0\in X$ and $c>0$.


\section{Canham-Helfrich functional}\label{sec:3}

We devote this section to recalling some results for the stationary case. In particular, by restricting to $d=3$ and $m=2$ we follow \cite{BLS:20} and introduce the generalization of the classical Canham-Helfrich energy \eqref{eq:ECH} to oriented curvature two-varifolds
\begin{equation}\label{eq:FCH}
F_\CH\colon AV_2^o(\RR^3)\to\RR,\qquad
F_\CH(V):=\int_{\RR^3\times\SS^2} f_\CH\big(\nu,A^{q_\sharp V}(q(x,\nu))\big)\,\dint V(x,\nu),
\end{equation} 
where the  integrand $f_\CH\colon\SS^2\times\RR^{3\times 3\times 3}\to\RR$ is given by
$$
{f_\CH(\nu,A):=\sum_{i=1}^3\left(\frac{\beta}{2}\Big(\sum_{j=1}^3A_{jij}-\nu_iH_0\Big)^2
	+\frac{\gamma}{2}\Big(\sum_{j=1}^3A_{jij}\Big)^2
	-\frac{\gamma}{4}\sum_{j,k=1}^3 A_{ijk}^2\right).}
$$ 
The form of the integrand in  \eqref{eq:FCH} can be deduced from  the identity $(H-H_0)^2=|\ove H-\nu H_0|^2=\sum_{i=1}^3(\ove H_i-\nu_i H_0)^2$ by replacing the mean curvature vector $\ove H=H\nu$  by its {\it varifold analogue}
$$
\textstyle{\ove H^{q_\sharp V}:=\left(\sum_{j=1}^3 A_{jij}^{q_\sharp V}\right)_{i=1}^3}.
$$ 
Note that the squared norm of the second fundamental form of a curvature varifold is given by
$
|\II^{q_\sharp V}|^2=\frac{1}{2}|A^{q_\sharp V}|^2=\frac{1}{2}\sum_{i,j,k=1}^3 (A_{ijk}^{q_\sharp V})^2,
$ 
which allows us to define the {\it generalized Gauss curvature}
\begin{equation}\label{eq:K}
K^{q_\sharp V}:=\frac{1}{2}\left(|\ove H^{q_\sharp V}|^2-|\II ^{q_\sharp V}|^2\right)=\frac{1}{2}|\ove H^{q_\sharp V}|^2-\frac{1}{4}|A^{q_\sharp V}|^2.
\end{equation}
Since $q(x,\nu)=(x,P)=q(x,-\nu)$, it follows that $(x,\nu)\mapsto A^{q_\sharp V}(q(x,\nu))$ does not depend on the sign of $\nu$ and reversing orientation has the sole effect of changing the sign in front of $H_0$. 

We identify the {\it area} of an oriented varifold $V\in\Radon(\RR^3\times\SS^2)$ with its mass, namely, $\mu_V(\RR^3)=V(\RR^3\times\SS^2)$. For $m_0>0$ fixed,  at the stationary level we are interested in varifolds $V\in AV_2^o(\RR^3)$ minimizing $F_\CH$ under the mass constraint  
\begin{equation}\label{eq:massconstraint}
\mu_V(\RR^3)=m_0.
\end{equation}

\subsection{Varifold minimizers} 

Existence of varifold minimizers follows  by the  Direct Method, exploiting compactness results for oriented and curvature varifolds (see Theorems  \ref{thm:IVmcpt} and \ref{thm:AVmcpt}) together with the lower semicontinuity \cite[Thm.\ 3.2]{BLS:20},
\begin{equation*}
F_\CH(V)\leq\liminf_{n\to\infty}F_\CH(V_n)
\quad\text{for}\quad
V_n\wto^\ast V \:\: (n\to\infty)\quad\text{in}\quad AV_2^o(\RR^3),
\end{equation*}  
where $V_n\wto^\ast V$ is the curvature varifold convergence, namely, $V_n\wto^\ast V$ in $\Radon(\RR^3\times\SS^2)$ and $(q_\sharp V_n,A^{q_\sharp V_n})\wto^\ast (q_\sharp V,A^{q_\sharp V})$ as measure-function pairs. A key ingredient in the mentioned lower-semicontinuity proof is the convexity of the integrand $f_\CH(\nu,.)$, which holds under the following condition on the parameters:
\begin{equation}\label{eq:betagamma}
-{\frac{6}{5}}\,\beta<\gamma<0.
\end{equation}
In fact, such convexity gives rise to the {\it curvature bound} \cite[Prop.\ 3.1]{BLS:20}
\begin{equation}\label{eq:varL2bounds}
\|A^{q_\sharp V}\|^2_{L^2_{q_\sharp V}(\RR^3\times G_{2,3})}\leq c_1\left(F_\CH(V)+c_2\,\mu_{V}(\RR^3)\right)
\end{equation}
for all $V\in AV_2^o(\RR^3)$, with constants $c_1>0$, $c_2\geq 0$ depending on data (in particular, $c_2=0$ if $H_0=0$). An early consequence of  \eqref{eq:varL2bounds} is that the Canham-Helfrich energy $F_\CH$ is bounded from below in terms of the mass.


\subsection{Lower bound in terms of the Willmore energy} 

A special instance of the classical Canham-Helfrich energy is the {\it Willmore energy} 
$$
\Will(M):=\frac{1}{4}\int_{M}H^2\dint\Haus^2,
$$
corresponding indeed to the choice $H_0=0$,  $\beta=1/2$, and $\gamma=0$  in \eqref{eq:ECH}. If $M\sbs\RR^3$ is an immersed closed surface one has $\Will(M)\geq 4\pi$ with equality only for the  sphere $M=\SS^2_R$ of any radius $R>0$ \cite{Willmore:1996}.  

Without changing notation, one can define the Willmore energy  of  an (oriented) curvature varifold $V\in AV_2^o(\RR^3\times\SS^2)$ as
$$
\Will(V):=\frac{1}{4}\int_{\RR^3\times G_{2,3}}|\ove H^{q_\sharp V}(x,P)|^2\dint (q_\sharp V)(x,P)
=\frac{1}{4}\int_{\RR^3\times\SS^2}|\ove H^{q_\sharp V}(q(x,\nu))|^2\dint V(x,\nu),
$$
which for $V=V[M,\nu^M,\theta^+,\theta^-]$ reduces to
\begin{equation}\label{eq:Will}
\Will(V)=\frac{1}{4}\int_{\RR^3}|\ove H^{q_\sharp V}|^2\dint\mu_V
=\frac{1}{4}\int_{M}|\ove H^{q_\sharp V}|^2(\theta^++\theta^-)\dint\Haus^2.
\end{equation}
The latter controls the multiplicity of the varifold $q_\sharp V=q_\sharp V[M,\theta^++\theta^-]$ via the {\it Li-Yau inequality} \cite{LiYau:82} (see \cite{KuSc:12} for its generalization to integral varifolds),
\begin{equation}\label{eq:LiYau}
\Will(V)\geq 4\pi (\theta^++\theta^-).
\end{equation}
The algebraic estimate 
$
|\ove H^{q_\sharp V}|^2=\sum_{i=1}^3\left(\sum_{j=1}^3A^{q_\sharp V}_{jij}\right)^2 
\leq 3\sum_{i,j,k=1}^3\left(A^{q_\sharp V}_{ijk}\right)^2=3|A^{q_\sharp V}|^2
$, 
cf.\ \cite[Lemma 2.3]{BLS:20}, entails $4\Will(V)\leq 3\|A^{q_\sharp V}\|^2_{L^2_{q_\sharp V}(\RR^3\times G_{2,3})}$. From this, under assumption \eqref{eq:betagamma}, the curvature bounds \eqref{eq:varL2bounds} imply the control
\begin{equation}\label{eq:WillFCHbound}
4\Will(V)\leq 3\,c_1\left(F_\CH(V)+c_2\,\mu_{V}(\RR^3)\right),
\end{equation}
which, in combination with \eqref{eq:LiYau},  in turn gives a  multiplicity bound in terms of the Canham-Helfrich energy, i.e.,  a Li-Yau-type inequality.
Let us recall again  that Li-Yau inequalities for the Canham-Helfrich functional have been recently obtained in \cite{RuSc:22}, where, however, the Gauss term is not considered.

In the following, we derive an alternative estimate in terms of the Willmore energy with explicit constants.
Recall 
that $V=V[M,\nu^M,\theta^+,\theta^-]$ is given by
\begin{equation}\label{eq:Mthetapm}
 \dint V(x,\nu)=\dint(\Haus^2\llc M)(x)\otimes (\theta^+(x)\delta_{\nu^M(x)}(\nu)+\theta^-(x)\delta_{-\nu^M(x)}(\nu)),
\end{equation}
with an $\Haus^2$-measurable, countably rectifiable $M\sbs\RR^3$ with orientation $\pm\nu^M(x)$ (corresponding to $(T_xM)^\pm$) and locally integrable multiplicities $\theta^\pm(x)\in\NN$ for $(\Haus^2\llc M)$-a.e.\ $x\in\RR^3$. Moreover, if $V$ is an (oriented) curvature varifold, i.e., $V\in AV_2^o(\RR^3\times\SS^2)$,  then the energy \eqref{eq:FCH} can be written as
\begin{align*}
 F_\CH(V)&=\int_{\RR^3\times\SS^2}\left(\frac{\beta}{2}\left|\ove H^{q_\sharp V}(q(x,\nu))-\nu H_0\right|^2+\gamma\, K^{q_\sharp V}(q(x,\nu)) \right)\dint V(x,\nu)\\
 &=\int_{M}\Bigg(\left(\frac{\beta}{2}\left|\ove H^{q_\sharp V}(q(x,\nu^M(x)))-\nu^M(x)H_0\right|^2+\gamma\, K^{q_\sharp V}(q(x,\nu^M(x))) \right)\theta^+(x)\\
 &+\left(\frac{\beta}{2}\left|\ove H^{q_\sharp V}(q(x,-\nu^M(x)))+\nu^M(x)H_0\right|^2+\gamma\, K^{q_\sharp V}(q(x,-\nu^M(x))) \right)\theta^-(x)\Bigg)\dint \Haus^2(x).
 \end{align*}
Taking into account that $q(x,-\nu^M)=q(x,\nu^M)$, 
$
|\ove H^{q_\sharp V}\mp\nu^MH_0|^2=|\ove H^{q_\sharp V}|^2\mp 2\ove H^{q_\sharp V}\cdot \nu^MH_0+|\nu^M|^2H_0^2
$, 
and $|\nu^M|^2=1$, this reduces to 
\begin{align}
 F_\CH(V)
 &=\int_M\left(\frac{\beta}{2}\left(\left|\ove H^{q_\sharp V}\right|^2+H_0^2\right)+\gamma\, K^{q_\sharp V}\right)(\theta^++\theta^-)\dint\Haus^2
 \nn \\
 &\qquad -\beta\,H_0\int_M \ove H^{q_\sharp V}\cdot \nu^M (\theta^+-\theta^-)\dint\Haus^2,
   \label{eq:FCHfull}
 \end{align}
where the argument $q(\cdot,\nu^M(\cdot))$ of $\ove H^{q_\sharp V}$ and $K^{q_\sharp V}$ has been omitted for conciseness. This representation of the Canham-Helfrich energy allows deducing
another lower bound of $F_\CH$ in terms of the Willmore energy.

\begin{lemma}[Lower bound for the energy]\label{lem:lowerbound}  
Let  $\beta>0$ and $\gamma$, $H_0\in\RR$. Then, for varifolds $V=V[M,\nu^M,\theta^+,\theta^-]\in AV_2^o(\RR^3)$  with mass $\mu_V(\RR^3)=m_0>0$, 
\begin{equation}\label{eq:lowerbound} 
F_\CH(V)\geq 2\beta\Big({\textstyle{\sqrt{\frac{\Will(V)}{m_0}}-\frac{|H_0|}{2}}}\Big)^2m_0+\gamma\int_M K^{q_\sharp V}(\theta^++\theta^-)\dint\Haus^2.
\end{equation}
\end{lemma}

\begin{proof}
Along the proof, we use the short-hand notation $H=\ove H^{q_\sharp V}\cdot\nu^M$ and omit everywhere the superscript $(\cdot)^{q_\sharp V}$ for brevity. Moving from \eqref{eq:FCHfull} we get
\begin{align*}
 F_\CH(V)
 &=\int_M\left(\frac{\beta}{2}\left(H^2+H_0^2\right)+\gamma K\right)(\theta^++\theta^-)\dint\Haus^2-\beta H_0\int_MH(\theta^+-\theta^-)\dint\Haus^2\\
\nn 
&=2\beta\left(\frac{1}{4}\int_MH^2(\theta^++\theta^-)\dint\Haus^2
-\frac{H_0}{2}\int_MH(\theta^+-\theta^-)\dint\Haus^2
+\frac{H_0^2}{4}m_0\right)\nn\\
& \quad+\gamma\int_M K(\theta^++\theta^-)\dint\Haus^2,\nn
\end{align*}
where we used the mass constraint $\mu_V(\RR^3)=\int_M(\theta^++\theta^-)\dint\Haus^2=m_0$ \eqref{eq:massconstraint} for the second equality. As $\Will(V)=\frac{1}{4}\int_MH^2(\theta^++\theta^-)\dint\Haus^2$ from \eqref{eq:Will} and $\theta^+,\theta^-\geq 0$, the second term in the right-hand side can be estimated as 
\begin{align*}
&\frac{H_0}{2}\int_MH(\theta^+-\theta^-)\dint\Haus^2
\leq |H_0|\int_M\frac{|H|}{2}|\theta^+-\theta^-|\dint\Haus^2
\leq |H_0|\int_M\frac{|H|}{2}(\theta^++\theta^-)\dint\Haus^2\\
& \qquad\leq |H_0|\sqrt{\frac{1}{4}\int_MH^2(\theta^++\theta^-)\dint\Haus^2}
\sqrt{\int_M(\theta^++\theta^-)\dint\Haus^2}
= |H_0|\sqrt{\Will(V)}\sqrt{m_0}\, .
\end{align*}
This gives the lower bound
$$
F_\CH(V) \geq 2\beta\left(\Will(V)-|H_0|\sqrt{\Will(V)}\sqrt{m_0}+\frac{H_0^2}{4}m_0\right)
+\:\gamma\int_M K(\theta^++\theta^-)\dint\Haus^2,
$$
 where the first term  can be rewritten as
$$
\textstyle{ \Will(V)-|H_0|\sqrt{\Will(V)}\sqrt{m_0}+\frac{H_0^2}{4}m_0=\Big(\sqrt{\Will(V)}-\frac{|H_0|}{2}\sqrt{m_0}\Big)^2
=\Big(\sqrt{\frac{\Will(V)}{m_0}}-\frac{|H_0|}{2}\Big)^2m_0}\, ,
$$
which completes the proof.
\end{proof}

\section{Varifold setting} \label{sec:4}

We now turn to the evolutionary case and consider the gradient flow of the Canham-Helfrich energy in the setting of curvature varifolds as in \cite{BLS:20}.   In Subsection \ref{ssec:existence}, we show global existence of a solution   using the GMM approach. We discuss restricted flows in Subsection \ref{ssec:restricted}  and, finally,  provide diameter bounds in Subsection \ref{ssec:diameter}.

To ease notations, the evolution will first be constrained to take place in a given compact set $\comp\sbs\sbs\RR^3$. 
By fixing the mass $m_0>0$ we let
\begin{equation}\label{eq:BV}
\cV:=\left\{V\in\Radon(\RR^3\times \SS^2): \: \mu_V(\RR^3)=V(\RR^3\times \SS^2)=m_0,\:{\rm \supp}(V)\sbs \comp\times\SS^2\right\},
\end{equation}
where we recall that ${\mathcal M}(\RR^3\times\SS^2)$ indicates the set of Radon measures on $\RR^3\times\SS^2$. We prescribe the class $\adm$ of {\it admissible varifolds} by 
\begin{align}\label{eq:admissible}
\adm:=\left\{V\in\cV:   \: V\in AV_2^o(\RR^3),\:\d(q_\sharp V)=0, \: \d\curr{V}=0\right\}. 
\end{align}
The class $\adm$ is closed in $AV_2^o(\RR^3)$ with respect to curvature varifold convergence. Indeed, let $V_n\in\adm$ and $V\in AV_2^o(\RR^3)$  be  such that $V_n\wto^\ast V$ in $AV_2^o(\RR^3)$ as $n\to\infty$. One readily gets that $\mu_V(\RR^3)=m_0$, ${\rm supp}(\mu_V)\sbs \comp$. Moreover, curvature varifold convergence implies that $\d(q_\sharp V_n)\wto^\ast\d(q_\sharp V)$ in $\Radon^3(\RR^3\times G_{2,3})$, which with $\d(q_\sharp V_n)=0$ gives the condition $\d(q_\sharp V)=0$. 
As shown in {\cite[Lemma 4.1; Eq.~(2.4)]{BLS:20}}, the limit $V$ also
satisfies $\d\curr{V}=0$. 

Let us note  that the requirement $\d\curr{V}=0$ in $\adm$ could be relaxed to  $\MM_{\d\curr{V}}(\RR^3)\leq m^c$ for some fixed $m^c\geq 0$, not affecting the analysis. We restrict ourselves to the choice $m^c=0$, however, for it is well-adapted to the modeling of single-phase membranes without boundary.  Moreover, it simplifies the notation and it will be   useful in the case of multiply covered smooth surfaces  (Section \ref{sec:regular}),  where we can take advantage of Federer's Constancy Theorem~
\ref{thm:constancy}.

In order to enforce the evolution in $\adm$, we additionally constrain the energy by defining
\begin{equation*}
{G_\CH}\colon\cV\to (-\infty,\infty],\quad
{G_\CH}(V):=\begin{cases} F_\CH(V),\quad  V\in {\adm}\\ 
\infty, \qquad\quad\:\: \text{else.} \end{cases}
\end{equation*}
We now endow $\cV$ with the Wasserstein metric $W_p$. To this aim, we fix $p \in [1,\infty)$ and follow the construction explained in  Subsection \ref{ssec:Wasserstein} with  $(X,d_X) = (\RR^3 \times \SS^2, d)$ where $d$ is the  metric defined as   
$$
d((x,\nu),(\tilde x,\tilde \nu)):=|x-\tilde x|+|\nu -\tilde\nu|
$$
(other equivalent metrics on $\RR^3 \times \SS^2$ could be considered as well). Note that $W_p(V,\tilde V)<\infty$ for all $V,\,\tilde V\in\cV$, for the supports of $\mu_V$ and $\mu_{\tilde V}$ are contained in the bounded set $\comp$. We have that 
$$
W_p^p( V,\tilde V)=
\left\{
  \begin{array}{ll}
   \displaystyle\min_{\lambda\in\Pi(V,\tilde V)}\displaystyle\int_{(\comp \times \SS^2)^2}(|x-\tilde x|
    +|\nu-\tilde\nu|)^p\dint\lambda((x,\nu),(\tilde x,\tilde\nu)) &\quad \text{for} \ p>1,\\[6mm]
    \sup\left\{{\displaystyle\int_{\comp \times \SS^2} f\,\dint(V-\tilde V):\:f\in\Lip_1(\comp\times \SS^2)}\right\}&\quad \text{for} \ p=1.
  \end{array}
  \right.
  $$
The actual choice of $p$ is irrelevant for the analysis and can be possibly adjusted to different modeling or computational needs.

Note that the Wasserstein metric $W_p(V,\tilde V)$ controls the distance of the varifolds $V$ and $\tilde V$ both in  space position and in orientation. Such strong control is necessary, for two varifolds sharing the same support in $\RR^3$ could still differ in their  orientation. In the more regular case of multiply covered smooth surfaces,  we will be able to consider a weaker metric as well, featuring just the Wasserstein distance of the supports of $\mu_V$ and $\mu_{\tilde V}$ in $\RR^3$, see  Section \ref{sec:regular}.

\subsection{Generalized Minimizing Movements}\label{ssec:existence}

We are now ready to introduce our concept of evolution.  Given an {\it initial state} $V^0\in\cV$ with $G_\CH(V^0)<\infty$  and  a time step $\tau>0$, we recursively define   {\it discrete minimizers} $(V^n_\tau)_{n\in \NN_0}$ by
$V^0_\tau := V^0$ and $V^n_\tau \in {\rm argmin}_\cV\, G_{\CH,\tau}(\cdot;V^{n-1}_\tau)$ for $n \in \NN$, 
where $G_{\CH,\tau}(\cdot;V^{n-1}_\tau)\colon\cV \to (-\infty,\infty]$ is the {\it incremental functional}
\begin{equation}\label{eq:incfunc}
G_{\CH,\tau}(V;V^{n-1}_\tau):= G_\CH(V) + \frac{1}{2 \tau } W_p^2(V,V^{n-1}_\tau).
\end{equation}
Define $\ove V_\tau\colon[0,\infty)\to \cV$ to be the piecewise constant interpolant given by $\ove V_\tau(0):=V^0$ and $\ove V_\tau(t):=V_\tau^n$ for $t\in((n-1)\tau,n\tau]$. The choice of the power $2$ in the Wasserstein term above corresponds to the case of {\it gradient flows}. Replacing $W_p^2(V,V^{n-1}_\tau)$ by $W_p^p(V,V^{n-1}_\tau)$ would require just minor notational adaptations and would correspond to the case of {\it doubly nonlinear flows} instead \cite{AmGiSa:08}. 

We define our evolution notion as follows,  see \cite[Def.\ 2.0.6]{AmGiSa:08}.

\begin{definition}[Generalized Minimizing Movements]\label{def:GMM}
A curve $V\colon[0,\infty)\to\cV$ is  called  a {\em Generalized Minimizing Movement (GMM)  for the varifold Canham-Helfrich flow}  if there exists a sequence $\tau_m \to 0$  as $m\to\infty$,  such that  $\ove V_{\tau_m}(t)\to V(t)$ for all $t\in [0,\infty)$. 
\end{definition}

A curve $V\colon[0,\infty)\to\cV$ is said to be {\it locally $L^2$-absolutely continuous} if there exists $w\in L^2_\loc(0,\infty)$ such that 
\begin{equation}
W_p(V(t_1),V(t_2))\leq\int_{t_1}^{t_2}w(t)\dint t \label{eq:metric_der}
\end{equation}
for all  $0\leq t_1< t_2<\infty$.  In this case, we write $V\in AC^2_\loc([0,\infty);\cV)$. Given a locally $L^2$-absolutely continuous curve $V$, the {\em metric derivative} 
$$
|V'|(t):=\lim_{s\to t}\frac{W_p (V(s),V(t))}{|s-t|}
$$
can be proved to exists for almost every $t \geq 0$ \cite[Ch.\ 1]{AmGiSa:08}. The function $[0,\infty)\ni t  \mapsto |V'|(t)$ belongs to $ L^2_\loc(0,\infty)$ and is minimal among the functions $w$ fulfilling \eqref{eq:metric_der}.

We start by recording a first existence result.
 

\begin{theorem}[Existence of GMMs]\label{thm:existence}
Assume \eqref{eq:betagamma} and let $V^0\in\cA$ with $G_\CH(V^0)<\infty$. Then, there exists a Generalized Minimizing Movement $V\in AC_{\loc}^2([0,\infty);\cV)$ associated to $G_\CH$ starting from $V(0)=V^0$. Moreover, $G_\CH(V(t))\leq G_\CH(V^0)$ and $V(t)\in\cA$ for all $t\geq 0$.
\end{theorem}

\begin{proof} The functional $G_\CH\colon\cV\to(-\infty,\infty]$ is defined on the complete metric space $(\cV,W_p)$ and has domain $\adm$. By assuming \eqref{eq:betagamma}, namely  $-6\beta/5<\gamma<0$, the Canham-Helfrich integrand $f_\CH(\nu,\cdot)$ is strictly convex for all $\nu\in\SS^2$ and the curvature bound \eqref{eq:varL2bounds}
$$
\|A^{q_\sharp V}\|^2_{L^2_{q_\sharp V}(\RR^3\times G_{2,3})}\leq c_1\left(G_\CH(V)+c_2\,\mu_{V}(\RR^3)\right)
$$
holds. Owing to the fact that the $W_p$-metric topology and the weak-$\ast$ topology of varifolds are equivalent, we deduce from \cite[Thm.\ 3.2]{BLS:20} that $G_\CH$ is lower semicontinuous in $(\cV,W_p)$. On the other hand, Theorems  \ref{thm:IVmcpt} and \ref{thm:AVmcpt} ensure that the sublevel sets of $G_\CH$ are compact. We are hence in the setting of \cite[Prop.\ 2.2.3]{AmGiSa:08} and the existence of a GMM $V\in AC_{\loc}^2([0,\infty);\cV)$ ensues for all  $V^0\in \adm$ with $G_\CH(V^0)<\infty$. 
Moreover, the solution $V$ satisfies the energy inequality 
$$
G_\CH(V(t))+\frac{1}{2}\int_0^t|V'|^2(r)\dint r\leq G_\CH(V^0)
$$
for all $t\geq 0$. 
In particular, we have $G_\CH(V(t))\leq G_\CH(V^0)<\infty$, which in turn entails that $V(t)\in\cA$.
\end{proof}

By their definition, GMMs starting from a minimizer $V^0$ of the Canham-Helfrich functional are stationary, i.e., $V(t)=V^0$ for all $t\geq 0$ whenever $V^0\in\text{argmin}_\cV\, G_\CH$.

\subsection{Restricted flows}\label{ssec:restricted}

In the following, we will be interested in considering GMMs in more restrictive settings. This can be easily accommodated within the above theory by simply constraining the functional $G_\CH$ to a closed subset of the admissible class $\adm$. More precisely, we have the following corollary to Theorem \ref{thm:existence}.

\begin{corollary}[Existence of restricted GMMs]\label{cor:restricted}
Assume \eqref{eq:betagamma} and let $\tilde\adm\sbs\adm$ be closed with respect to the weak-$\ast$ topology. Then, there exists a GMM associated to
\begin{equation*}
  \tilde G_\CH (V):=\left\{
  \begin{array}{ll}
    G_\CH(V)&\quad \text{if} \ V \in \tilde A,\\
    \infty &\quad \text{else}
  \end{array}
\right.
\end{equation*}
starting form $V^0\in\tilde\adm$ with $V(t)\in\tilde\adm$ for all $t\geq 0$.
\end{corollary}

The main application of Corollary \ref{cor:restricted} is that of ensuring the existence of GMMs in the regular setting of multiply covered smooth surfaces in Section \ref{sec:regular}. In the remainder of this subsection, we apply the corollary to demonstrate the possibility of including a volume constraint or enforcing symmetry. 


\subsubsection{Volume constraint}\label{ssec:volume}

The modeling of biomembranes usually  requires 
to constrain the minimization of the Canham-Helfrich energy to surfaces with fixed enclosed volume. This can be accomplished in the varifold setting as well.

We introduce the {\it enclosed volume} of an oriented varifold $V\in IV_2^o(\RR^3)$ by
\begin{equation}\label{eq:vol}
\encvol(V):=\frac{1}{3}\int_{\RR^3\times\SS^2}x\cdot\nu\,\dint V(x,\nu).
\end{equation}
This definition is inspired by the smooth setting: if $V=V[M,\nu^M,1,0]$ such that $M=\d E$ for a $C^{1}$ domain $E\sbs\RR^3$ with outer unit normal $\nu^M$, as $\div(x)=3$ the divergence theorem yields 
$$
\encvol(V)
=\frac{1}{3}\int_{M}x\cdot\nu^M(x)\,\dint \Haus^2(x)=\frac{1}{3}\int_E\div (x)\dint x=\int_E\dint x=|E|=:{\rm vol}(M).
$$
To include the volume constraint we restrict the varifold evolution to the set
$$
\tilde A := \{V\in\cA:\: \encvol(V)=v_0\}
$$
for some fixed volume $v_0>0$. In order not to rule out classical smooth solutions by violating the isoperimetric constraint, we require $36\pi v_0^2\leq m_0^3$.  As the functional $V\mapsto \encvol(V)$ from \eqref{eq:vol} is linear, the set $\tilde A$ above is closed with respect to the weak-$\ast$ convergence. Corollary \ref{cor:restricted} thus gives existence of restricted GMMs respecting a volume constraint. 

\subsubsection{Symmetry}

Starting form a symmetric initial configuration $V^0$, it is unclear if GMMs conserve symmetry. The lack of uniqueness of GMMs and the fact that minimizers of the Canham-Helfrich energy are not necessarily spheres suggest that symmetry conservation may not hold.

On the other hand, one can use Corollary \ref{cor:restricted}  to  enforce additionally symmetry by means of a  constraint. Let us firstly discuss the case of {\it reflection symmetry}. Without loss of generality, we assume that we reflect about a plane $P\sbs\RR^3$ through the origin
and that the  compact $\comp\sbs\sbs\RR^3$ in \eqref{eq:BV}  is symmetric with respect to $P$.
The associated reflection mapping
$\psi\colon \RR^3\times\SS^2 \to \RR^3\times\SS^2$  is defined  by 
$$
\psi(x,\nu)=(Sx,S\nu) 
 \qquad\text{with}\qquad S\in O(3),\quad \det S=-1.
$$
The reflection of $V\in\Radon(\RR^3\times\SS^2)$ is the pushforward measure $\psi_\sharp V\in \Radon(\RR^3\times\SS^2)$. As reflection is linear, the class of reflection-symmetric admissible varifolds, 
$$
\tilde \cA := \{V\in\adm:\: \psi_\sharp V=V\},
$$
is weakly-$\ast$ closed in $\adm$. One can hence apply Corollary \ref{cor:restricted} and obtain that, starting from a reflection-symmetric configuration $V_0 \in \tilde\cA$, there exists a reflection-symmetric GMM $V(t)\in \tilde \cA $ for all $t \geq 0$. 

The case of {\it rotational symmetry} around a fixed axis $(x_a,\nu_a)\in \RR^3\times\SS^2$ can be treated analogously. Without loss of generality, we choose $x_a$ to be the $x_3$-axis in $\RR^3$ and $\nu_a$ to be the north pole in $\SS^2$, i.e.,  $\nu_a=(0,0,1)$. Moreover, we assume that $\comp$ is rotationally symmetric around $x_a$. Let $\psi^\alpha\colon \RR^3\times\SS^2 \to \RR^3\times\SS^2$ given by 
$$
\psi^\alpha(x,\nu)=(R^\alpha x,R^\alpha\nu) \qquad\text{with}\qquad R^\alpha\in SO(3)
$$
denote the anticlockwise rotations about an angle $\alpha\in[0,2\pi)$ around the axes $x_a$, $\nu_a$ respectively. A measure $V\in\Radon(\RR^3\times\SS^2)$ is rotationally symmetric if for all $\alpha\in[0,2\pi)$ there exists a rotation $\psi^\alpha$ such that $\psi_\sharp^\alpha V=V$. 
Any rotation $R^\alpha\in SO(3)$ can be represented as composition of two reflections, namely $R^\alpha = S_1\circ S_2$ for $S_i\in O(3)$ with $\det S_i=-1$ ($i=1,2$), so that $\psi^\alpha=\psi_1\circ\psi_2$ for $\psi_i(x,\nu):=(S_ix,S_i\nu)$, and hence the class of rotationally symmetric varifolds,
$$
\tilde \cA:=\{V\in\adm:\: \psi_\sharp^\alpha V=V\:\text{for all}\:\alpha\in[0,2\pi)\},
$$
can be readily checked to be weakly-$\ast$ closed as well. Therefore, Corollary \ref{cor:restricted} applies, showing existence of rotationally symmetric GMMs, starting from a rotationally symmetric $V_0 \in \tilde\cA$.


\subsection{Diameter bounds}\label{ssec:diameter}

Although our notion of Canham-Helfrich flow is rather weak, we can  show that solutions do not degenerate, for the GMMs of Definition \ref{def:GMM} do not shrink to points nor expand indefinitely.

\begin{prop}[Diameter bounds]\label{l:dia}
Assume \eqref{eq:betagamma} and let $t\mapsto V(t)$ be a GMM. Then, there exist constants $c_0,c^0>0$ depending only on the parameters $(V^0,m_0,\beta,\gamma,H_0)$, such that
$$
 c_0\leq\diam(\supp(\mu_{V(t)}))\leq c^0\quad \forall\: t \geq 0.
 $$
\end{prop}

\begin{proof}
We start from the lower bound. By \cite[Lemma 4.12]{BLS:20}, an oriented curvature varifold $V\in AV_2^o(\RR^3)$ satisfies 
$$
\diam(\supp(\mu_V))\geq\frac{\mu_V(\RR^3)}{\sqrt{\mu_V(\RR^3)\,\Will(V)}+\|\d (q_\sharp V)\|_{\TV}(\RR^3\times G_{2,3})}.
$$
As $V\in\adm$, we have $\mu_V(\RR^3)=m_0$ and $\d(q_\sharp V)=0$, so that the right-hand-side is bounded below by
$\sqrt{\frac{m_0}{\Will(V(t))}}$.
Moreover, $\Will(V(t))\leq \frac{3}{4}\, c_1\left(F_\CH(V(t))+c_2\,m_0\right)$ by \eqref{eq:WillFCHbound}. As a GMM satisfies $V(t)\in\adm$ and $F_\CH(V(t))=G_\CH(V(t))\leq G_\CH(V^0)$ for all $t\geq 0$ by the energy-decreasing property, this implies 
$\Will(V(t))\leq \frac{3}{4}\, c_1\left(G_\CH(V^0)+c_2\,m_0\right)$. Consequently we obtain the lower bound
$$
\diam(\supp(\mu_{V(t)}))\geq\sqrt{\frac{m_0}{\frac{3}{4}\, c_1\left(G_\CH(V^0)+c_2\,m_0\right)}}:=c_0>0.
$$
The starting point for the upper bound is the generalization to curvature varifolds without boundary of the estimate 
$\diam(\supp(\mu_V))\leq\frac{2}{\pi}\sqrt{\mu_V(\RR^3)\,\Will(V)}$, \cite{Simon:93, Topping:98}.
Then, similar arguments as for the lower bound show that
\begin{align*}
\diam(\supp(\mu_{V(t)}))
& \leq \frac{1}{\pi}\sqrt{3\,m_0\, c_1\left(G_\CH(V^0)+c_2\,m_0\right)}=:c^0>0
\end{align*}
for all $t\geq 0$, proving the claim.
\end{proof}

Note in particular that the upper bound $c^0$ is independent of the choice of the bounding compact $\comp$. This proves that restricting to $\comp$ is in fact not necessary. This allows us to omit $\comp$ in the sequel,  i.e., we replace $\cV$ introduced in \eqref{eq:BV} by $\Prob(\RR^3\times\SS^2)$, see Subsection \ref{ssec:Wasserstein}.

\section{Regular setting}\label{sec:regular}

Let us now turn our attention to the case of multiply covered
surfaces, which amounts to constrain the evolution to the class of varifolds with smooth supports in $\RR^3$. Such more regular setting allows for finer results. In fact, one can prove that the multiplicity is preserved along the evolution. On this basis, one has the possibility of considering a weaker setting where the Wasserstein distance between varifolds, which is rather strong, is replaced by the more natural Wasserstein distance between supports
in space.

We start by introducing some notation in Subsection \ref{ssec:mul}. Then, we discuss the reference case of multiply covered spheres in Subsection \ref{ssec:spheres}. In particular, we assess the minimality of multiply covered spheres in dependence of the given  mass constraint   $m_0$ and spontaneous curvature $H_0$. We then present a Li-Yau-like bound on the multiplicity in terms of the Canham-Helfrich energy for general multiply covered surfaces in Subsection \ref{ssec:LiYau}. In order to possibly consider a restricted flow in the spirit of Corollary \ref{cor:restricted} for multiply covered surfaces, we verify closedness
of the corresponding classes in Subsection \ref{ssec:closure}. An argument for multiplicity conservation during the flow is presented in Subsection \ref{ssec:mult}. Eventually, we discuss the case of the weaker evolution setting given by the Wasserstein  distance  between the supports in Subsection \ref{ssec:wassweaker} and we comment on some generalizations in Subsection \ref{ssec:dalphin}.


\subsection{Multiply covered surfaces}\label{ssec:mul}

In the remainder of the paper, we restrict to varifolds $V=V[M,\nu,\theta^+,\theta^-]$ that are supported on closed (i.e., compact without boundary), connected
$C^{1,1}$ surfaces $M\sbs\RR^3$, that is, embedded $C^1$-surfaces $M$ with Lipschitz continuous unit normal vector $\nu:=\nu^M\colon M\to\SS^2$. In this setting, by Rademacher's theorem, $\ove H^{q_\sharp V}$, $|A^{q_\sharp V}|^2$, and $K^{q_\sharp V}$ are defined almost everywhere  on $M$.

By the admissibility condition on the zero boundary current in the Definition \eqref{eq:admissible} of $\cA$, i.e., $\d\curr{V}=0$, Federer's constancy theorem for smoothly supported currents (Theorem \ref{thm:constancy}) entails that the density $\theta^+-\theta^-$ of $\curr{V}$ is constant. This gives $\overline{\theta}_0\in\ZZ$ such that $(\Haus^2\llc M)$-a.e.,
$$
\theta^+-\theta^-=\overline{\theta}_0.
$$
In this setting, the Canham-Helfrich functional \eqref{eq:FCHfull} reduces to
$$
F_\CH(V)=\int_M\left(\frac{\beta}{2}\left(H^2+H_0^2\right)+\gamma K\right)(\theta^++\theta^-)\dint\Haus^2-\beta\,\overline{\theta}_0\,H_0\int_MH\dint\Haus^2.
$$
It follows that $F_\CH(V)$ will be small for large $\overline{\theta}_0$, provided that $H_0\int_MH\dint\Haus^2>0$. This favorable case occurs for example if $H$ does not change sign (i.e.,\ $M$ is mean convex) and $H_0H>0$.
Otherwise, if $H_0 \int_MH\dint\Haus^2<0$, a negative multiplicity $\overline{\theta}_0$ is favored, corresponding to an orientation flip.
 
This suggests to restrict with no loss of generality to multiplicities of the form
$$
\theta^+=\overline{\theta}_0=k\in\NN
\qquad\text{and}\qquad \theta^-=0.
$$
In order to ensure the closure of this more regular class of configurations, we follow \cite{Dalphin:18} and further restrict to surfaces of {\it uniform} $C^{1,1}$ regularity with given genus $g\in\NN_0$. More precisely, fix $L>0$ and a smooth connected closed surface $\Sigma^g$ embedded in $\RR^3$ and with genus $g\in\NN_0$. Then, for $k\in\NN$, we consider the subclass of admissible varifolds given by 
\begin{align}\label{eq:admgk}
\admgk:=\Big\{& V=V[M,\nu,k,0]\in\adm :
  \:\text{$M=\psi(\Sigma^g)$ for an embedding $\psi\in C^{1,1}(\Sigma^g;\RR^3)$ }\nonumber\\
  &\text{with $\|\psi\|_{C^{1,1}(\Sigma^g)}\leq L$, $\:\nu$ is the outer unit normal vector to $M$}\Big\}.
\end{align}
We will also  employ  the notation
\begin{equation}\label{eq:admg}
\admg:=\{V\in\adm:\, \exists\:k\in\NN\:\:\text{such that}\:V\in\admgk\}.
\end{equation}
to indicate multiply covered surfaces of genus $g$.  Recall that a differentiable map $\psi\colon\Sigma\to\RR^3$ is an embedding if it is injective and $\nabla\psi$ has rank two. Moreover, 
$$
\|\psi\|_{C^{1,1}(\Sigma)}:=\sup_{x\in\Sigma}\left(|\psi(x)|+|\nabla\psi(x)|+\sup_{x'\neq x\in\Sigma}\frac{|\nabla\psi(x)-\nabla\psi(x')|}{|x-x'|}\right).
$$
The constraint $\|\psi\|_{C^{1,1}(\Sigma^g)}\leq L$ entails that elements $V=V[M,\nu,k,0]\in\admgk$ have a uniform curvature bound of their spatial supports $M\sbs\RR^3$. Equivalently, the sets $M$ satisfy a uniform ball condition or are sets of positive reach \cite{Federer:59, Dalphin:18}. 

For  $V\in\admgk$ we have that
\begin{align*}
\hspace{-1em}
G_\CH(V)=F_\CH(V)&=k\,E_\CH(M),
\end{align*} 
with
$$
E_\CH(M)=\int_M\left(\frac{\beta}{2}(H-H_0)^2+\gamma K\right)\dint\Haus^2
=\frac{\beta}{2}\int_M(H-H_0)^2\dint\Haus^2+4\pi\gamma (1-g)
$$
by the Gauss-Bonnet theorem. 
Note that the existence of smooth closed surfaces $M$ of fixed
genus $g\in\NN$ minimizing  $E_\CH$  can be proved in some parameter
range \cite{ChVe:13, Dalphin:18, MoSc:20, RuSc:22}.

\subsection{Stationarity of spheres}\label{ssec:spheres}

Before moving on, let us collect some remarks on the case where $M$ is a sphere. We thus restrict to the class $\admzk$ in the following. As GMMs starting from a Canham-Helfrich energy minimizer  are stationary, we investigate
under which conditions we have minimality of multiply covered spheres.

By $\SS_R^2$ we denote a sphere with radius $R>0$, centered at some $x_0\in\RR^3$, and equipped with its outer unit normal $\nu\colon \SS_{R}^2\to\SS^2$. In particular, $\SS_R^2$ has the constant mean curvature $H=-2/{R}$ and the constant Gauss curvature $K=1/{R^2}$. 
A {\it multiply covered sphere} is an oriented varifold of the form  
$V_{\SS^2_R}:=V[\SS^2_R,\nu,\theta_0,0]$ with multiplicity $\theta_0\colon\SS^2_R\to\NN$. Its mass is given by ${\mu_{V_{\SS^2_R}}(\RR^3)=\int_{\SS^2_R}\theta_0\dint\Haus^2=m_0>0}$ and
its energy \eqref{eq:FCH} takes the form
\begin{align}
F_\CH(V_{\SS^2_R})
&=\frac{\beta}{2}\left(-\frac{2}{R}-H_0\right)^2m_0+\frac{\gamma}{R^2}m_0
=2\beta m_0\left(\frac{1}{R}+\frac{H_0}{2}\right)^2+\frac{\gamma}{R^2}m_0\nn\\
&=\left(\frac{2\beta+\gamma}{R^2}+\frac{2\beta H_0}{R}+\frac{\beta H_0^2}{2}\right)m_0
=2\beta m_0\left(\frac{1+\frac{\gamma}{2\beta}}{R^2}+\frac{H_0}{R}+\frac{H_0^2}{4}\right).
\label{eq:FCHspheremult}
\end{align}
In case of a constant multiplicity $\theta_0(x)=k\in\NN$ for a.e.\
$x\in\SS^2_R$, the mass of $V_{\SS^2_R}$ reads
$$
\mu_{V_{\SS^2_R}}(\RR^3)=4\pi k R^2=m_0,
$$
which relates the radius and the multiplicity. Consequently, for fixed mass, the only multiply covered sphere (up to isometries) with constant multiplicity $k\in\NN$ is the {\it $k$-covered sphere}
\begin{equation*}\label{eq:spheremult}
 S_k:=V_{\SS^2_{R_k}}=V[\SS^2_{R_k},\nu,k,0]
\qquad\text{with}\qquad R_k:=\sqrt{\frac{m_0}{4\pi k}}=\frac{R_1}{\sqrt{k}},\qquad R_1:=\sqrt{\frac{m_0}{4\pi}}.
\end{equation*}
Since $\SS_{R_k}^2$ is just a scaling of $\SS^2$, we have
$
S_k\in\admzk,
$
at least for $k$ small with respect to the uniform curvature bound  $L$  in \eqref{eq:admgk}.  The energy of $S_k$  is a function of the multiplicity $k$ alone,  
\begin{equation}\label{eq:FCHk}
F_\CH(S_k)
=2\beta\, m_0\left(\frac{1+\frac{\gamma}{2\beta}}{R_1^2}k+\frac{H_0}{R_1}\sqrt{k}+\frac{H_0^2}{4}\right).
\end{equation}
For $\beta=1/2$, $H_0=0$, and $\gamma=0$, this reduces to the Willmore energy of a $k$-covered sphere,
$\textstyle{\Will(S_k)=\frac{1}{4}\big(\frac{2}{R_k}\big)^2m_0=4\pi
  k}$, which corresponds to equality in the Li-Yau estimate
\eqref{eq:LiYau}.


The following lemma shows that the $k$-covered sphere is the only Canham-Helfrich minimizer among multiply covered surfaces of genus zero, provided that $|H_0|$ is small enough. 

\begin{prop}[Spheres are unique minimizers for fixed multiplicity]\label{lem:multiplesphere} 
Let $\beta>0$, $\gamma\in\RR$, and $k\in\NN$ such that $S_k\in\admzk$. If $H_0\leq 0$ and satisfies 
$$
|H_0|\leq 2/R_{k}=\sqrt{16\pi k/m_0},
$$ 
then $S_k$ uniquely minimizes $F_\CH$ on $\admzk$. The statement is also true if $H_0>0$, if $S_k$ is defined with the inner normal. In case $H_0=0$, $S_k$ is a minimizer independently of its orientation.
\end{prop}

\begin{proof}
Let $V \in \admzk$ be arbitrary. Starting  from the lower bound  \eqref{eq:lowerbound} given in Lemma \ref{lem:lowerbound}, 
we aim at deducing a new lower bound that is independent of $V$.
Since $\theta^++\theta^-=k\in\NN$ and $M$ is a closed $C^{1,1}$ surface of genus $g=0$, the Gauss-Bonnet theorem applies:
\begin{equation}\label{eq:Gauss}
\gamma\int_M K^{q_\sharp V}(\theta^++\theta^-)\dint\Haus^2=\gamma\,k\int_M K^{q_\sharp V}\dint\Haus^2=4\pi\,\gamma\,k(1-g)=4\pi\,\gamma\, k.
\end{equation}
Combining the assumption
$\frac{|H_0|}{2}\leq \frac{1}{R_k}= \sqrt{\frac{4\pi k}{m_0}}$ 
with $\sqrt{\frac{\Will(V)}{m_0}}\geq \sqrt{\frac{4\pi k}{m_0}}$, from  the Li-Yau estimate \eqref{eq:LiYau}, shows that 
$\sqrt{\frac{\Will(V)}{m_0}}-\frac{|H_0|}{2}\geq \sqrt{\frac{4\pi k}{m_0}}-\frac{|H_0|}{2}\geq 0$.
With  \eqref{eq:lowerbound} and \eqref{eq:Gauss}, this gives 
\begin{equation*}\label{eq:lowerboundnew}
F_\CH(V)\geq 2\beta m_0\Big({\textstyle{\sqrt{\frac{\Will(V)}{m_0}}-\frac{|H_0|}{2}}}\Big)^2+\gamma\int_M K^{q_\sharp V}(\theta^++\theta^-)\dint\Haus^2
\geq 2\beta m_0\Big({\textstyle{\sqrt{\frac{4\pi k}{m_0}}-\frac{|H_0|}{2}}}\Big)^2+4\pi\,\gamma\,k.
\end{equation*} 
This lower bound is uniquely attained at $S_k$. In fact, by  \eqref{eq:FCHspheremult} and $H_0\leq 0$ we get 
\begin{equation*}
\textstyle{F_\CH(S_k)=2\beta m_0\left(\frac{1}{R_k}+\frac{H_0}{2}\right)^2+\frac{\gamma}{R_k^2}m_0
=2\beta m_0\left(\sqrt{\frac{4\pi k}{m_0}}-\frac{|H_0|}{2}\right)^2+4\pi\,\gamma\,k.}
\end{equation*}
This proves that $S_k$ is the unique minimizer in $\admzk$,  because equality in the Li-Yau estimate \eqref{eq:LiYau}  holds only  for spheres.  Choosing $\nu$ as the inner normal just reverses the sign in front of $H_0$ and the assertion follows as above. Finally, we note that for $H_0=0$, the assumption $|H_0|/2\leq 1/R_k$ is trivially satisfied.
\end{proof}


Proposition \ref{lem:multiplesphere} is an extension  of an argument by \textsc{Dalphin} \cite[Chapter 5]{Dalphin:14} to the multilayered setting. There, also counterexamples to sphere minimality are explicitly constructed in classes of axisymmetric surfaces (cigars or stomatocytes), if the bound on $H_0$ is violated. Of course, the same counterexamples apply  to our  multiply covered case. 

Let us now drop the prescription of the multiplicity while still asking $0\leq -H_0\leq \sqrt{16\pi/m_0}$, so that any minimizer of $F_\CH$ in $\admz$ \eqref{eq:admg} is necessarily a sphere. Our aim is then to determine which sphere is actually the absolute minimizer across multiplicities. We emphasize that due to the possibly nontrivial multiplicity, the Gauss-term \eqref{eq:Gauss}
is variable as well and, unless $\gamma=0$, can not be neglected in the optimization process. We have the following.

\begin{prop}[Optimal spheres]\label{lem:optspheres} 
Let $-2\beta<\gamma\leq 0$, $0\leq -H_0\leq \sqrt{16\pi/m_0}$, and define 
\begin{align}
k_\ast & :=\frac{m_0H_0^2}{16\pi}\,\frac{1}{(1+\frac{\gamma}{2\beta})^2}\,, \label{eq:k}\\
Y_\ast& :=(\sqrt{\lfloor k_\ast \rfloor}-\sqrt{\lceil k_\ast \rceil}\,)\,H_0 - \textstyle{\sqrt{\frac{4\pi}{m_0}}\left(1 + \frac{\gamma}{2\beta} \right)}.\label{eq:Y}
\end{align}
If $k_\ast\in\NN$, then the unique minimizer of $F_\CH$ on $\cA^0$ is given by $S_{k_\ast}$. If $k_\ast\in [0,1]$ then the unique minimizer is $S_1$. If $k_\ast\in(1,\infty)\setminus\NN$, then the unique minimizer is $S_{\lfloor k_\ast \rfloor}$ for  $Y_\ast<0$ and $S_{\lceil k_\ast \rceil}$ for  $Y_\ast>0$. In case $k_\ast\in(1,\infty)\setminus\NN$ and $Y_\ast=0$, the two spheres $S_{\lfloor k_\ast \rfloor}$ and $S_{\lceil k_\ast \rceil}$ are the sole minimizers. 
\end{prop}

Before giving the proof of the proposition, let us note that the assumption $-2\beta<\gamma\leq 0$ implies $0<1 + \frac{\gamma}{2\beta}\leq 1$. Moreover, under the smallness condition on $|H_0|$ of the statement, one has that $k_* \leq (1+\frac{\gamma}{2\beta})^{-2}$. In particular, the multiplicity of the optimal sphere is bounded above in terms of the parameters $\beta$ and $\gamma$ only. Yet, the determination of the actual minimizing sphere $S_k$ via the sign of $Y_\ast$ does depend on the parameters $m_0$ and $H_0$ as well. 

\begin{proof}[Proof of Proposition \ref{lem:optspheres}]
Looking back to 
\eqref{eq:FCHk},
we have that $k \mapsto F_\CH(S_k)$ is the restriction to $\NN$ of the smooth, strictly convex, and unbounded function  $f\colon (0,\infty)\to \RR$ given by
$$
 f(r) : 
 = 2\beta m_0\left(\frac{1+\frac{\gamma}{2\beta}}{R_1^2}r+\frac{H_0}{R_1}\sqrt{r}+\frac{H_0^2}{4}\right).
$$
As $f$ is uniquely minimized at $r=k_\ast$, one has that $\NN\ni k \mapsto F_\CH(S_k)$ is minimized either in $k=k_\ast$ if  $k_\ast\in\NN$, in $k=1$ if $0\leq k_\ast\leq 1$, and otherwise in $k=\lfloor k_\ast \rfloor$, in $k=\lceil k_\ast \rceil$, or in both of the latter. These assertions follow directly by comparing $F_\CH(S_{\lfloor k_\ast \rfloor})$ and $F_\CH(S_{\lceil k_\ast \rceil})$ via the sign of
$Y_\ast=\frac{R_1}{2\beta m_0}\left(F_\CH(S_{\lfloor k_\ast \rfloor})-F_\CH(S_{\lceil k_\ast \rceil})\right)=\frac{R_1}{2\beta m_0}\left(f(\lfloor k_\ast \rfloor)-f(\lceil k_\ast \rceil)\right)$, which by $R_1=\sqrt{\frac{m_0}{4\pi}}$ and $\lfloor k_\ast \rfloor-\lceil k_\ast \rceil=-1$ takes the form \eqref{eq:Y}. 
\end{proof}

Let us close this subsection by discussing the specific case of the
single-covered sphere. We have the following.

\begin{corollary}[Optimality of the single-covered sphere]\label{cor:una} 
Let $-2\beta<\gamma\leq 0$ and
$$
\textstyle{0\leq -H_0<\sqrt{\frac{4\pi}{m_0}}\left(1 + \frac{\gamma}{2\beta}\right)(1+\sqrt{2}).}
$$
Then, $S_1$ is the unique minimizer of $F_\CH$ on $\cA^0$. 
\end{corollary}
 
 Indeed,  under the condition on $H_0$ one readily checks that $-H_0< \sqrt{16\pi/m_0}$ as well as that $k_\ast$ from \eqref{eq:k} and $Y_\ast$ from \eqref{eq:Y} fulfill $k_\ast 
 <2$ and $Y_\ast<0$, respectively. Hence, the assertion of Corollary \ref{cor:una} follows from Proposition \ref{lem:optspheres}.

\subsection{Li-Yau estimate}\label{ssec:LiYau}
We now derive a multiplicity bound in terms of the Canham-Helfrich energy, reminiscent of the Li-Yau inequality 
$
4 \pi k\leq \Will(V)
$ 
\eqref{eq:LiYau}
for the Willmore energy of an integral varifold 
$V=V[M,k]$. The estimate below can be compared with the recent \cite{RuSc:22}, where some alternative Li-Yau inequality is derived for the case $\gamma=0$.

We restrict to varifolds $V$ in the class $\admgk$ \eqref{eq:admgk}, for we make a crucial use of the Gauss-Bonnet theorem. We have the following.  

\begin{prop}[Multiplicity bound]\label{lem:kmax}
Let $V\in\admgk$, $H_0\in\RR$, $\beta>0$, $\gamma<0$, and $-2\beta<\gamma(1-g)$. Then, there exists a nondecreasing function $\RR\ni F \mapsto \ove k(F) \in \NN $ depending on $\beta$, $\gamma$, and $H_0$, but independent of the genus $g$, such that 
$
k\leq \ove k({F_\CH(V)})
$.
\end{prop}

The proposition ensures that the multiplicity of any $V\in\admg$ can be bounded in terms of its Canham-Helfrich energy, as long as $0<-2\beta<\gamma(1-g)$. Note that the latter holds for any genus $g\in\NN_0$ if  $-2\beta<\gamma<0$ or \eqref{eq:betagamma} holds. In fact, with \eqref{eq:betagamma}, i.e., $-\frac{6}{5}\beta<\gamma<0$, we also have $-2\beta<\gamma<0$, because $-2\beta=-\frac{5}{3}\frac{6}{5}\beta<\frac{5}{3}\gamma<\gamma<0$. This immediately gives  $-2\beta<\gamma(1-g)$ in the cases $g=0$ and $g=1$, namely $-2\beta<\gamma$ and $-2\beta<0$ respectively. If $g>1$ then $1-g<0$, so that $\gamma<\gamma(1-g)$ for $\gamma<0$. Therefore $-2\beta<\gamma<0$ implies $-2\beta<\gamma(1-g)$ in this case as well.

\begin{proof}[Proof of Proposition \ref{lem:kmax}] By Lemma \ref{lem:lowerbound}, any oriented curvature varifold $V=V[M,\nu,\theta^+,\theta^-]\in AV_2^ o(\RR^3)$ with mass $\mu_V(\RR^3)=m_0$ satisfies
$$
F_\CH(V) \geq 2\beta\left(\Will(V)-|H_0|\sqrt{\Will(V)}\sqrt{m_0}+\frac{H_0^2}{4}m_0\right)
+\:\gamma\int_M K(\theta^++\theta^-)\dint\Haus^2,
$$
where the Gauss curvature $K=K^{q_\sharp V}$ is given by \eqref{eq:K}. For $V\in\admgk$, we have 
$\theta^+=k$, $\theta^-=0$, and $M$ is a $C^{1,1}$ surface, for which Gauss-Bonnet applies:
$
\int_M K(\theta^++\theta^-)\dint\Haus^2=k\int_M K\dint\Haus^2=4\pi (1-g)k
$. By Young's inequality, 
$|H_0|\sqrt{\Will(V)}\sqrt{m_0}\leq\veps\Will(V)+\frac{H_0^2m_0}{4\veps}$
with $\veps>0$. By employing the classical Li-Yau inequality $\Will(V)\geq 4\pi k$, the estimate then gives
\begin{align*}
F_\CH(V) &\geq 2\beta\left((1-\veps)\Will(V)+\left(1-\frac{1}{\veps}\right)\frac{H_0^2m_0}{4}\right)
+\:4\pi\gamma(1-g)k\\
 &\geq 4\pi\left(2\beta(1-\veps)+\gamma(1-g)\right)k+\beta\left(1-\frac{1}{\veps}\right)\frac{H_0^2m_0}{2}.
\end{align*}
Alternatively, one can deduce this lower bound from 
$F_\CH(V)=k E_\CH(M)$ and $k \Haus^2(M)=m_0$, starting form the analogue of Lemma \ref{lem:lowerbound} in the classical case.

Thanks to the assumptions $\beta>0$, $\gamma<0$, and $2\beta+\gamma(1-g)>0$, we can choose 
$$
0<\veps<\frac{1}{2\beta}(2\beta+\gamma(1-g))=1+\frac{\gamma(1-g)}{2\beta}
\begin{cases}=1+\frac{\gamma}{2\beta}, & g=0,\\
=1, & g=1,\\
>1, & g\geq 2,\end{cases}
$$ 
such that $2\beta\veps<2\beta+\gamma(1-g)$. Then
$2\beta(1-\veps)+\gamma(1-g)=2\beta+\gamma(1-g)-2\beta\veps>0$ and we obtain the multiplicity bound
$$
k\leq\frac{F_\CH(V)-\beta\left(1-\frac{1}{\veps}\right)\frac{H_0^2m_0}{2}}{4\pi \left(2\beta(1-\veps)+\gamma(1-g)\right)}.
$$ 
With the choice $\veps=\frac{1}{4\beta}(2\beta+\gamma(1-g))$ for $g=0$ or $1$ and $\veps=1$ for $g\geq 2$
we get
$$
k\leq \begin{cases}
\frac{1}{4\pi}\left(\frac{2}{2\beta+\gamma}F_\CH(V)+\beta\frac{2\beta-\gamma}{(2\beta+\gamma)^2}H_0^2m_0\right)=:\ove k_0, & g=0,\\
\frac{1}{4\pi}\left(\frac{1}{\beta}F_\CH(V)+\frac{H_0^2m_0}{2}\right)=:\ove k_1, & g=1,\\
\frac{F_\CH(V)}{4\pi\,\gamma(1-g)}\leq -\frac{F_\CH(V)}{4\pi\,\gamma}=:\ove k_2, & g\geq 2.\\
\end{cases}
$$ 
The inequality in case $g\geq 2$ follows from $\gamma<0$, by which $\gamma(1-g)\geq -\gamma$, and the fact that $F_\CH(V)\geq 0$ in this case. Thus, the next largest integer of the maximum of these values gives the desired upper multiplicity bound that is increasing in $F_\CH(V)$,
$$
\ove  k({F_\CH(V)}):=\lceil \max\{\ove k_0, \ove k_1,  \ove k_2\}\rceil.
$$
We note that $ \ove k({F_\CH(V)})$ depends on $\beta$, $\gamma$, and $H_0$ but not on the genus $g$.
\end{proof}

\subsection{Closedness of regular classes}\label{ssec:closure}

We now check that Corollary \ref{cor:restricted} can be applied to the subclass of admissible varifolds corresponding to multiply covered smooth surfaces. In particular, this amounts to  proving that the classes $\admg$ \eqref{eq:admg} and $\admgk$ \eqref{eq:admgk} are closed  in $\adm$ \eqref{eq:admissible}.  We have the following.  
\begin{lemma}[Closedness]\label{lem:closed} Let $(V_n)_{n\in\NN}\subset\adm$ and $V\in\adm$ with $V_n\wto^\ast V$ as $n\to\infty$.
\begin{enumerate}
    \item[{\em (i)}] If $(V_n)_{n\in\NN}\subset\admg$ then $V\in\admg$.
    \item[{\em (ii)}] Let $k\in\NN$. If $(V_n)_{n\in\NN}\subset\admgk$ then $V\in\admgk$.
\end{enumerate}
\end{lemma}

\begin{proof}
We first prove (ii). Let $V_n\in\admgk$ and $V\in\adm$ be such that $V_n\wto^\ast V$ as varifolds. By definition, $V_n=V_n[M_n,\nu^{M_n},k,0]$ and by compactness of the uniform $C^{1,1}$ surfaces $M_n$ \cite{Federer:59,Dalphin:18}, there exists a uniform $C^{1,1}$ surface $M$ such that (a not relabeled subsequence satisfies) $M_n\to M$ ($n\to\infty$) in the Hausdorff distance. This gives rise to the varifold 
$$
\tilde V:=\tilde V[M,\nu^{M},k,0]\in\admgk.
$$
By uniqueness of varifold limits $\tilde V$ coincides with $V$ (which also implies that the whole sequence $M_n$ converges). This shows that $V\in\admgk$, completing the proof of (ii).

In order to show (i), we observe that by \eqref{eq:admg},  elements $V_n\in\admg$ are of the form $V_n=V_n[M_n,\nu^{M_n},j_n,0]$ for some $j_n\in\NN$. The mass constraint implies that 
$$
j_n=\frac{m_0}{\Haus^2(M_n)}.
$$
As before, $M_n\to M$ ($n\to\infty$) in the Hausdorff sense. The uniform $C^{1,1}$ regularity implies that oscillations are controlled, and $\Haus^2(M_n)\to\Haus^2(M)$. Therefore, we have
$$
j_n=\frac{m_0}{\Haus^2(M_n)}\to \frac{m_0}{\Haus^2(M)}=:j.
$$
This gives rise to the varifold $\tilde V=\tilde V[M,\nu^{M},j,0]$, which must coincide with $V$, as in the proof of (ii), so $V\in\admg$.
\end{proof}

\subsection{Conservation of multiplicity}\label{ssec:mult}

Corollary \ref{cor:restricted} and Proposition \ref{lem:closed} ensure that, starting from an initial state $V^0\in\admgk$, there exists a restricted GMM  for the varifold Canham-Helfrich flow  in $\admgk$.
In this section, we prove a stronger fact, namely, that multiplicity is conserved by GMMs in $\cA^g$, without the need to be constrained. 


\begin{theorem}[Conservation of multiplicity]\label{thm:mult} 
Assume \eqref{eq:betagamma} and $V^0\in\admgk$ with $G_\CH(V^0)<\infty$. Let $V$ be a GMM restricted to $\admg$. Then $V(t)\in\admgk$ for all $t\geq 0$, i.e., the flow conserves multiplicity.
\end{theorem}

\begin{proof}
Let us start by remarking that, thanks to Proposition \ref{lem:kmax}, any  $V\in\admg$ actually satisfies $V\in\admgk$ for some $k\leq \ove k ({G_\CH(V)})$. Let $V$ be a restricted  GMM in $\admg$ with $V(0)=V^0$. We know by the closedness of $\admg$ from Lemma \ref{lem:closed} (i)  and by Corollary \ref{cor:restricted} that such GMMs exist. Let $(V^n_m)_{m,n\in\NN}\sbs\admg$ be a corresponding sequence of discrete minimizers with time step $ \tau_m>0$ for $m\in\NN$ with $\tau_m\to 0$. Then $G_\CH(V^0)\geq G_\CH(V^n_{m})$ and, since the function $\ove k$ from Proposition \ref{lem:kmax} is nondecreasing, it follows that
$$
\ove k({G_\CH(V^n_{m})})\leq \ove k({G_\CH(V^0)})=:\ove k_0.
$$
Next, we show that $V^n_{m}\in\admgk$ by induction on $n\in\NN$. Assume that $V^{n-1}_{m}\in\admgk$ and let by contradiction $V^n_{m}\in\admg_j$ with $j\neq k$. Lemma \ref{lem:distance} below ensures that  
\begin{equation}\label{eq:use}
W_p(V^{n}_{m},V^{n-1}_{m})\geq\alpha(j,k)>0.
\end{equation}
By defining  $\ove\alpha:=\min_{1\leq j\neq k\leq \ove k_0}\alpha(j,k)>0$ and $f_j:=\inf_{\admg_j}G_\CH$ we hence have
$$
G_\CH(V^0)\geq G_\CH(V^{n-1}_{m})\geq G_\CH(V^n_{m})+\frac{1}{2\tau _{m}}W_p(V^{n}_{m},V^{n-1}_{m})\geq f_j+\frac{\ove\alpha^2}{2\tau _{m}}.
$$
This leads to a contradiction as soon as the time step $\tau_{m}$ is
small enough, namely, if 
$$
\tau _{m}<\frac{\ove\alpha^2}{2(G_\CH(V^0)-f_j)}.
$$
Therefore, all discrete minimizers must satisfy $V^n _{m}\in\admgk$ for $\tau_m$ small enough. Of course, the same holds true for the piecewise constant interpolants $\ove V_{\tau_m}(t)$ for all $t\geq 0$. Finally, the closedness of $\admgk$ from Proposition \ref{lem:closed} (ii)  shows that multiplicity $k$ is preserved when taking the limit $\tau_m\to 0$  as well, i.e.,  $\ove V_{\tau_m}(t)\to V(t)\in\admgk$ for all $t\geq 0$, which
completes the proof. 
\end{proof}

We are hence left with checking the following technical Lemma, which is used in the argument of Theorem \ref{thm:mult} for formula \eqref{eq:use}. In particular, we show a uniform  control from below of the Wasserstein distance between two multiply covered surfaces with different multiplicities.

\begin{lemma}[Distance threshold]\label{lem:distance} For all $j,\,k\in\NN$ with $j \not = k$ there exists $\alpha=\alpha(j,k)>0$ such that 
$$
W_p(V_j,V_k)\geq \alpha
$$ 
for all $V_j\in\admg_j$ and $V_k\in\admgk$.
\end{lemma}


\begin{figure}
  \centering
\pgfdeclareimage[width=70mm]{zzeronb}{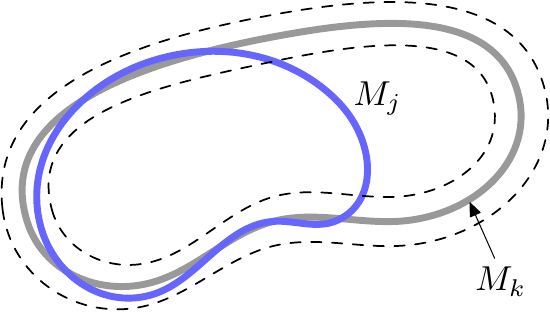}
\pgfuseimage{zzeronb}
  \caption{Illustration of the argument of Lemma \ref{lem:distance} for $j>k$. The surface $M_j$ (blue) must leave the tubular neighbourhood (dashed line) around the surface $M_k$ (grey).}
\label{fig:zzeronb}
\end{figure}

\begin{proof} As the statement is symmetric in $j$ and $k$, with no loss of generality we may assume that $j>k$. We proceed by contradiction and assume that for all $\delta>0$ there exist  $V_j\in\admg_j$ and  $V_k\in\admgk$ with 
\begin{equation}
W_p(V_j,V_k)<\delta^2.\label{eq:condition}
\end{equation}
We show that condition \eqref{eq:condition} leads to a contradiction if $\delta$ is chosen to be small enough. More precisely, we prove that one can find $0<\delta_*<1$ depending only on $k$ and $j$, as well as on the data $m_0$, $L$, and $p$, such that for any $0<\delta<\delta_*$ no varifolds  $V_j\in\admg_j$ and $V_k\in\admgk$  exist such that \eqref{eq:condition} holds.  

Along the proof, we use the notation $V_j=V[M_j,\tilde \nu,j,0]$ and $V_k=V[M_k,\nu,k,0]$ for any varifold in $\admg_j$ and  $\admgk$, respectively. Moreover, we indicate by $C_i >0$, $i=0,\dots, 7$, some constants depending solely on the data, in particular possibly on $L$. To start with, let us note that there exists $C_0>0$ such that any $V \in \admg$ is such that the corresponding spatial support $M$ in $\RR^3$ fulfills the lower bound 
\begin{equation*}
  {\rm vol} (M)\geq C_0, 
\end{equation*}
where  ${\rm vol} (M)$ indicates the volume of the region which is enclosed by $M$, cf.\ Subsection \ref{ssec:volume}.

We subdivide the proof into three steps. At first, we show that any surface which is completely contained in a sufficiently small neighborhood of $M_k$ must have at least the area of $M_k$, up to a small error. Since 
$$
\Haus^2(M_j)<\Haus^2(M_k)
$$ 
by $j>k$ and by the mass constraint,  this implies
that $M_j$ cannot be completely contained in such a small neighborhood of $M_k$, see Figure \ref{fig:zzeronb}. 
In the second step, we give a lower bound on the area of $M_j$ which is not contained in the small neighborhood of $M_k$. Eventually, in Step 3 we show that such a lower bound contradicts condition \eqref{eq:condition}. 

{\it Step 1: A lower bound on the area.} 
We start by checking that one can find $\delta_0>0$ and $C_1>0$ such that   
\begin{align}
   &\forall\, 0< \delta < \delta_0, \ \forall\, V_k\in\admgk,\ \forall\,  E\subset \RR^3\ \text{of finite perimeter with} \, \ |E|\geq C_0 : \nonumber\\
  & \quad  \partial^*E \subset
  M_k^\delta : = M_k+ B_\delta(0) \ \Rightarrow \ 
    \Haus^2(\partial^*E) \geq \Haus^2(M_k) -C_1\delta,  \label{eq:monster}
\end{align}
where we recall the notation for the open ball $ B_\delta(0)=\{x\in\RR^3: \: |x|<\delta\}$.

Let us start by choosing $\delta_1>0 $ so small that the projection $\pi_{M_k}$ onto $M_k$ is well-defined on $M_k^{\delta_1}$, for any $V_k \in \admgk$. As $M_k$ is a set of positive reach, such a $\delta_1$ exists and depends on $L$ only. Then, one can find $C_2>0$ such that 
\begin{equation}\label{eq:the2}
{\rm vol} (M_k^\delta) \leq C_2 \delta \Haus^2(M_k) \quad \forall\, 0<\delta<\delta_1.
\end{equation}
Note that $C_2$ depends on $L$ and that the latter inequality cannot be guaranteed if the uniform $C^{1,1}$ regularity of $M_k$ is dropped.
 Next we choose $0<\delta_0 <\min\{ C_0k/(C_2 m_0),\delta_1\}$, which implies that 
\begin{equation*}
 \partial^*E \subset M_k^{\delta_0}  \ \Rightarrow \ \pi_{M_k}(\partial^*E)\equiv M_k.
\end{equation*}
Indeed, if  $\pi_{M_k}(\partial^*E)\not = M_k$, the fact that $\partial^*E \subset M_k^{\delta_0} $ would imply that the whole $E$ is contained in $M_k^{\delta_0}$. In combination with inequality \eqref{eq:the2} this would require that
$$
C_0 \leq |E| \leq {\rm vol} (M_k^{\delta_0}) \leq C_2 {\delta_0} \Haus^2(M_k) =C_2 {\delta_0} \frac{m_0}{k}<C_0,
$$
which is a contradiction.

Let $\xi\colon U\to \Sigma^g$ be a fixed smooth a.e.\ injective parametrization of $\Sigma^g\subset \RR^3$ with $U=[0,1]^2$. We can cover $U$ up to a negligible set by the union of disjoint open sets $\{U_i\}_{i=1}^{N_\eta}$ on which $\xi$ is injective, each having diameter smaller than some given $\eta>0$, to be fixed below. Note that if $\eta$ is small enough, such a covering can be chosen in such a way that  $N_\eta\leq 2/\eta^2$  (take disjoint open squares of side $\eta$, where we have $N_\eta\leq \lceil 1/\eta  \rceil^2$, hence $N_\eta\leq 2 /\eta^2$ for $0<\eta<(\sqrt{5}-1)/2\:$).  

Given $V_k \in \admgk$ with $M_k = \psi(\Sigma^g)$ as in \eqref{eq:admgk}, define $\phi:=\psi\circ \xi\colon U\to M_k$ and recall that the normal at $x\in M_k$ is
indicated by $\nu(x)$. We let $u_i \in U_i $ be fixed and define $P_i\subset \RR^3$ to be the tangent plane to $M_k$ at the point $x_i =\phi(u_i)$. Moreover, for all $0<\delta<\delta_0$ we define the $\delta$-neighborhood of $\phi(U_i)$ in the normal direction, namely,  
$$
D^{\delta}_i:= \{ \tilde x = x + h\,\nu(x) \in \RR^3: \: x\in
\phi(U_i), \: |h|<\delta \}, 
$$
see Figure~\ref{fig:Mk}. 
\begin{figure}
  \centering
\pgfdeclareimage[width=130mm]{Mkk}{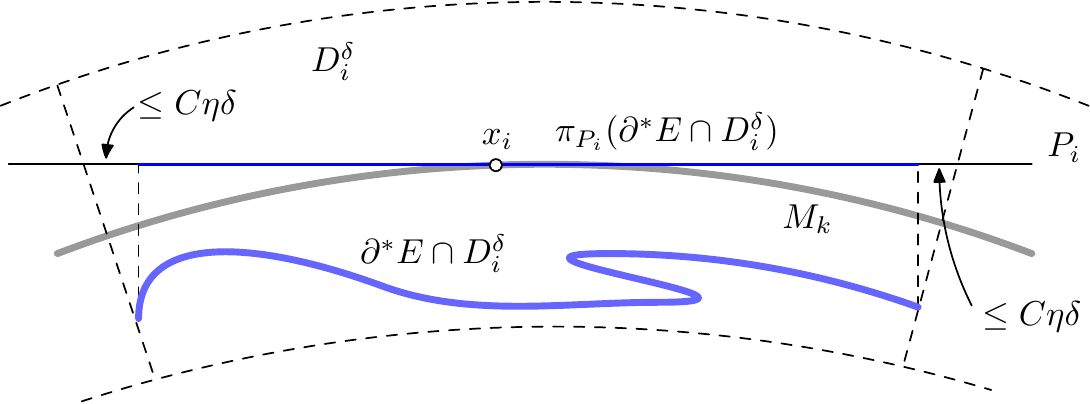}
\pgfuseimage{Mkk}
  \caption{The notation of {\it Step 1} in the proof of Lemma \ref{lem:distance}. }
\label{fig:Mk}
\end{figure}
We observe that 
\begin{equation}
  \label{eq:1}
  \Haus^2( \partial^*E \cap D^\delta_i)  \geq \Haus^2(\pi_{P_i}(\partial^*E \cap D^\delta_i)),
\end{equation}
for orthogonal projections to a plane decrease the perimeter. Next, we bound $\Haus^2(\pi_{P_i}(\partial^*E \cap D^\delta_i))$ by the area of $P_i\cap D^\delta_i$, up to a small error. More precisely, we exploit the Lipschitz continuity of $x \mapsto \nu(x)$ in order to get that 
\begin{equation}\sup_{y \in P_i\cap D^\delta_i}\ \inf_{ x \in \pi_{P_i}(\partial^*E \cap D^\delta_i)} |x-y|\leq C_3\eta\delta,
\label{eq:inq}
\end{equation}
where $C_3>0$ depends on $L$ only. 
By possibly choosing another constant $C_4>0$, again depending on $L$,  we can check that the length of the boundary $\partial ( P_i\cap D^\delta_i )\subset P_i$ is bounded by $C_4\eta$. Inequality \eqref{eq:inq} hence implies that 
\begin{equation}\label{eq:2}
  \Haus^2(\pi_{P_i}(\partial^*E \cap D^\delta_i)) \geq \Haus^2(P_i\cap
  D^\delta_i) - C_5\eta^2\delta,
\end{equation}
where $C_5>0$ is just depending on $L$.
Eventually, one can estimate the area of $P_i \cap D^\delta_i$ via that of $\phi(U_i)$. In fact, the surface $P_i \cap D^\delta_i$ can be parametrized over $U_i$ via 
$$
\tilde \phi(u) = \phi(u) + h(u)\nu(\phi(u))
$$
where the $C^1$ function $h\colon U_i \to (-\delta,\delta)$ is uniquely determined by asking $P_i \cap D^\delta_i = \tilde\phi(U_i)$. Owing to the smoothness of $M_k$ we can show that the infinitesimal area elements on $\phi(U_i)$ and $P_i \cap D^\delta_i$ are comparable.  For this purpose, we  introduce the short-hand notation  
$$ 
a:= \partial_1 \phi \times \partial_2
\phi\colon U_i \to \RR^3\quad \text{and}\quad \tilde a:= \partial_1 \tilde \phi
\times \partial_{ 2}\tilde \phi\colon U_i \to \RR^3.
$$
As by definition
$\tilde a = (\partial_1 \phi +(\partial_1 h) \nu + h (\nabla\nu) \partial_1\phi)\times (\partial_2 \phi +(\partial_2 h)  \nu + h (\nabla\nu) \partial_2\phi)$,
we get  
$$
a\cdot \tilde a = |a|^2+ b \cdot a
$$
for $b:= \partial_1 \phi \times h (\nabla\nu) \partial_2 \phi + h (\nabla\nu) \partial_1\phi \times \partial_2 \phi + h^2 (\nabla\nu) \partial_1 \phi \times
(\nabla\nu) \partial_2\phi$.
We can hence find $C_6>0$ such that $|\tilde a| \leq C_6$ and $|b| \leq C_6\delta$.  As $x \mapsto \frac{a(x)}{|a(x)|}=\nu(x)$ is Lipschitz continuous we get 
$
| \frac{a}{|a|} - \frac{\tilde a}{|\tilde a|}| \leq C_7
\eta.
$
The latter in particular entails that 
$$
C_7 \eta \geq  \left| \left(\frac{a}{|a|} - \frac{\tilde a}{|\tilde a|}\right)\cdot \frac{\tilde a}{|\tilde a|} \right|= \left|\frac{a \cdot \tilde a}{|a|\,|\tilde a|}-1 \right| = \left|\frac{|a|^2+b\cdot a}{|a|\,|\tilde a|}-1 \right|.
$$
By multiplying by $|\tilde a|$ we obtain
$
\left| |a|  - |\tilde a|+ \frac{b\cdot a}{|a|} \right| \leq  C_7 \eta |\tilde a|.
$
Eventually, by using the bounds $|\tilde
a|\leq C_6$ and $|b|\leq C_6 \delta$, we conclude that $$
\big||\partial_1 \phi \times \partial_2 \phi| - |\partial_1 \tilde
\phi \times \partial_{2}\tilde \phi|\big|\leq C_6 \delta + C_6 C_7 \eta.
$$
This implies 
\begin{align}
 \Haus^2(P_i\cap D^\delta_i) &= \int_{U_i}  |\partial_1 \tilde \phi\times \partial_{2}\tilde \phi|\geq \int_{U_i}\big(|\partial_1 \phi\times \partial_2 \phi| - (C_6 \delta + C_6 C_7 \eta)
  \big)\nonumber\\
&=\Haus^2(\phi(U_i))- (C_6 \delta + C_6 C_7 \eta) \Haus^2(U_i).  \label{eq:3}
\end{align}
By combining inequalities \eqref{eq:1}, \eqref{eq:2}-\eqref{eq:3}, and summing for $i=1,\dots,N_\eta$, we get
\begin{align*}
\Haus^2(\partial^*E) &\geq \Haus^2(M_k) -  C_5\eta^2\delta N_\eta- (C_6 \delta + C_6 C_7 \eta)\Haus^2(U)\\
&\geq \Haus^2(M_k) - 2 C_5 \delta - (C_6 \delta + C_6 C_7 \eta).
\end{align*}
It hence suffices to let $\eta = \delta$ and define $C_1:=2C_5+C_6+C_6C_7$ to obtain \eqref{eq:monster}.  

{\it Step 2: Uniform bound of the excess area. }
Given $V_j \in \admg_j$, we apply \eqref{eq:monster} to $E\subset \RR^3 $ with $\partial E =M_j$. Consequently,  if $M_j
\subset M_k^\delta$, then
$$ 
\frac{m_0}{j} = \Haus^2(M_j) \geq \Haus^2(M_k) -  C_1 \delta >\frac{m_0}{k} -  C_1 \delta_0
$$
 is true  for any $0<\delta <\delta_0$. Recalling $j>k$, this leads to a contradiction as soon as 
\begin{equation}\label{eq:deltamo}
   C_1 \delta_0  <\frac{m_0}{k}  - \frac{m_0}{j}. 
\end{equation}
Under the latter assumption, we have hence proved that $M_j \not\subset M_k^\delta $ for any $\delta< \delta_0$ and any $V_j \in \admg_j$ and $V_k \in \admgk$. 
 
We now aim at refining this argument by showing that, for $\delta$ small enough, the area of the portion of $M_j$ which is not contained in $M_k^\delta$ can be uniformly bounded from below. More precisely, given any bounded open set $A\sbs\RR^3$, we show that there exists $0<\delta_* <  \delta_0$ such  
that 
\begin{align}\label{eq:aux}
  &\forall\, V_k \in \admgk, \, \forall\, E \subset \RR^3 \ \text{with} \  {\rm Per} (E;A)\leq \frac{m_0}{j} , \ |E|\geq C_0: \quad {\rm Per} (E;(\overline{M_k^{\delta_*}})^c) \geq \delta_*. 
\end{align}

Statement \eqref{eq:aux} can be checked by contradiction: assume that for all $0<\delta< \delta_0$ there exists $E_\delta\subset\RR^3$ with ${\rm Per} (E_\delta;A)\leq m_0/j$, $ |E_\delta|\geq C_0$, and $ {\rm Per}(E_\delta;(\overline{M_k^{\delta}})^c) <\delta$. By compactness in the class of finite-perimeter sets we find a not relabeled subsequence $(E_\delta)_\delta$  and a finite perimeter set $E\subset\RR^3$ such that,  as $\delta\to 0$,  the characteristic functions $1_{E_\delta}$ converge to $1_E$ in $L^1(\RR^3)$ and the Radon measures $\mu_\delta := -\D 1_{E_\delta}$ fulfill $\mu_\delta \wto^\ast \mu = -\D 1_{E}$. This implies that ${\rm vol}(E)\geq  C_0$ and ${\rm Per}(E;A) \leq m_0/j$, as well. Moreover, for all $0<\overline\delta< \delta_0$ we have that
\begin{align*}
   {\rm Per}(E;(\overline{M_k^{\overline \delta}})^c) \leq \liminf_{\delta \to
  0}   {\rm Per}(E_\delta;(\overline{M_k^{\overline \delta}})^c) \leq \liminf_{\delta \to
  0}   {\rm Per}(E_\delta;(\overline{M_k^{  \delta}})^c) = 0.
\end{align*}
This shows that $\partial^* E \subset M_k^{\overline \delta}$. By applying estimate \eqref{eq:monster} we deduce that $ {\rm Per}(E;M^{\delta_1}_k) \geq \Haus^2(M_k) - C_1 \overline \delta$ which leads to a contradiction as
$$ 
\frac{m_0}{j} \geq {\rm Per}(E; M^{\delta_1}_k) \geq \Haus^2(M_k) - C_1 \overline \delta = \frac{m_0}{k} - C_1 \overline \delta\stackrel{\eqref{eq:deltamo}}{>} \frac{m_0}{j}.
$$
This proves that there exists $0<\delta_*<\delta_1$ such that \eqref{eq:aux} holds. In particular, for all $V_j\in \admg_j$ and $V_k\in \admgk$ we have that
\begin{equation}\label{eq:step2}
  \Haus^2 (M_j \cap  (\overline{M_k^{\delta_* }})^c) \geq \delta_* .
\end{equation}

{\it Step 3: Conclusion of the proof.} Let $V_j\in \admg_j$ and $V_k\in \admgk$ fulfill condition \eqref{eq:condition} for some $\delta <\delta_*$. Denote by $\lambda \in\Pi(V_k,V_j)\sbs\Radon((\RR^3\times\SS^2)^2)$ an optimal transport plan between $V_k$ and $V_j$, i.e.,
$$
W_p^p(V_j,V_k)=\int_{(\RR^3\times\SS^2)^2} d^p((x,\nu),(\tilde x,\tilde \nu))\,\dint\lambda((x,\nu),(\tilde x,\tilde\nu)),
$$
 where we recall that $d((x,\nu),(\tilde x,\tilde \nu))=|x-\tilde x|+|\nu-\tilde\nu|$.  Consider the ${\delta}$-neighborhood $N^\delta$ of the diagonal in
$(\RR^3\times\SS^2)^2$ given by
$$
N^\delta:=\left\{(y,\tilde y)\in (\RR^3\times\SS^2)^2:\:d(y,\tilde y)<{\delta}\right\}.
$$
As $1\leq \delta^{-p}d^p(y,\tilde y)$ for all $(y,\tilde y) \in
(N^\delta)^c$ we can estimate the measure of the complement  $(N^\delta)^c$ by a
 classical Chebyshev-type argument:
\begin{equation}\label{eq:Cheby}
\lambda((N^\delta)^c)=\int_{(N^\delta)^c}\dint\lambda\leq\delta^{-p} \int_{(\RR^3\times\SS^2)^2}d^p(y,\tilde y)\dint\lambda(\tilde y,y)  = \delta^{-p} W_p^p(V_j,V_k)\stackrel{\eqref{eq:condition}}{<}\delta^{p}.
\end{equation}
Next we observe that $\tilde x\in M_j\cap (\overline{M_k^\delta} )^c$ implies $|\tilde x-x|\geq {\delta}$ and thus $((x,\nu(x)),(\tilde x,\tilde\nu (\tilde x))) \in (N^\delta)^c$ for all $x\in M_k$. 
With $Z:=\{((x,\nu (x)),(\tilde x, \tilde \nu (\tilde x)))\in(\RR^3\times\SS^2)^2: \: x \in M_k,\, \tilde x \in M_j\cap (\overline{M_k^\delta}
)^c\} \subset  (N^\delta)^c $ we hence have that  
\begin{equation}\Haus^2(M_j\cap (\overline{M_k^\delta})^c) \leq \lambda(Z) 
  \leq \lambda((N^\delta)^c)\stackrel{\eqref{eq:Cheby}}{<}\delta^p.\label{eq:step3}
\end{equation} 
Eventually, we combine inequalities \eqref{eq:step2} and
\eqref{eq:step3} to get
$$ \delta >\delta^p \stackrel{\eqref{eq:step3}}{ > } \Haus^2(M_j\cap
 (\overline{M_k^\delta})^c)\geq \Haus^2(M_j\cap (\overline{M_k^{\delta_*}})^c)
\stackrel{\eqref{eq:step2}}{\geq}\delta_*>\delta$$
which is a contradiction.  
\end{proof}

\subsection{Existence of GMMs under a weaker metric}\label{ssec:wassweaker}

In the regular case, one can prove the existence of GMMs for the Canham-Helfrich functional under the weaker metric given by the Wasserstein distance of the spatial  supports.
Indeed, let $V^0,V^1\in\Prob(\RR^3 \times \SS^2)$ and set $\mu^0:=\mu_{V^0},\mu^1:=\mu_{V^1}\in\Prob(\RR^3)$. Then, one can prove that  the varifold Wasserstein distance,  $W_p(V^1,V^0)$, controls the Wasserstein distance between  spatial  Radon measures, $\ove W_p(\mu^1,\mu^0)$, in the following sense 
$$
W_p(V^1,V^0)\geq \ove  W_p (\mu^1,\mu^0).
$$
Indeed, with $\lambda\in\Pi(V^1,V^0)$ realizing the minimum in the definition of $W_p(V^1,V^0)$, we get
\begin{align*}
& W_p^p(V^1,V^0)=\int_{(\RR^3\times\SS^2)^2}(|x-\tilde x|+|\nu-\tilde\nu|)^p\dint\lambda((x,\nu),(\tilde x,\tilde\nu))\\
&=\int_{(\RR^3\times\SS^2)^2}|x-\tilde x|^p\dint\lambda((x,\nu),(\tilde x,\tilde\nu))+\underbrace{\sum_{k=1}^p\binom{p}{k}\int_{(\RR^3\times\SS^2)^2}|x-\tilde x|^{p-k}|\nu-\tilde\nu|^{k}\dint\lambda((x,\nu),(\tilde x,\tilde\nu))}_{\geq\: 0}\\
&\geq\int_{\RR^3\times \RR^3}|x-\tilde x|^p\dint(\pi_\sharp^{\RR^3}\lambda)(x,\tilde x)\underbrace{\int_{\SS^2\times\SS^2}\dint(\pi_\sharp^{\SS^2}\lambda)(\nu,\tilde\nu)}_{=\:1}
\geq\ove W_p^p(\mu^1,\mu^0),
\end{align*}
as the $\RR^3\times\SS^2$-marginals $V^0$ and $V^1$ of $\lambda$ are oriented varifolds of the form \eqref{eq:Mthetapm} and the $(\RR^3)^2$-marginal $\pi_\sharp^{\RR^3}\lambda\in\Radon((\RR^3)^2)$ is a possible choice in $\Pi(\mu^1,\mu^0)$.

On the other hand, as soon as multiplicity is  fixed to some $k\in\NN$ and we restrict to the regular varifolds $V^1,V^0\in\admgk$ \eqref{eq:admgk}, which have with fixed orientation,  we have that
\begin{equation}\label{eq:nondeg}
 \ove W_p(\mu^1,\mu^0)=0\:\Longrightarrow\:W_p(V^1,V^0)=0,
 \end{equation}
for $\ove W_p(\mu^1,\mu^0)=0$ in particular implies that $M^0=M^1=:M$. The equality $V^1=V^0$ then follows from $V^0 =V[M ,\nu^{M},k,0]=  V^1$ with $\mu_{V^0}=k(\Haus^2\llc M)$.

The nondegeneracy condition \eqref{eq:nondeg} is instrumental in proving the existence of GMMs, because it qualifies $\ove W_p$ as a distance on $\admgk$. Correspondingly, the existence result from \cite[Prop.~2.2.3]{AmGiSa:08} can still be applied. 
However, note that condition \eqref{eq:nondeg} hinges on the conservation of multiplicity and that the proof of Theorem \ref{thm:mult}, in particular the argument of Lemma \ref{lem:distance}, requires the control of the stronger distance $W_p(V^1,V^0)$. This amounts to say that, effectively, we can resort to the weaker metric $\ove W_p (\mu^1,\mu^0)$ just in the frame of  GMMs restricted a priori
 to $\admg$, namely, multiply covered uniformly regular $C^{1,1}$ surfaces.

\subsection{More general curvature functionals}\label{ssec:dalphin}

Most results of this section can  also be established 
for a larger class of geometric functionals. In particular, we could prove the existence of GMMs for more general curvature energies of the form
$$
D(M):= \int_M f(x,\nu(x),H(x),K(x))\dint\Haus^2(x).
$$
Here, $M$ is a uniformly regular $C^{1,1}$ surface in
$\RR^3$,  as in \eqref{eq:admgk}, and the integrand  $f$ is continuous and convex in its last two arguments. This class of curvature functionals has already been considered in \cite{Dalphin:18} concerning equilibrium problems. In particular, it has been verified there that $D$ admits minimizers in $\cA^g_1$ for any fixed genus $g\in \NN$, also under a volume constraint.

The analysis in \cite{Dalphin:18} hinges upon the compactness of $\cA^g_1$ and on the lower semicontinuity of $D$. These same tools ensure the existence of a GMM for $D$ with
respect to the Wasserstein metric $W_p$.
The adjustments required are mainly notational. The only point deserving some attention is the a priori multiplicity bound in terms of the functional, for a Li-Yau inequality may fail in degenerate cases. Still, a bound on the 
 multiplicity  can be directly deduced from the uniform $C^{1,1}$ regularity, i.e., depending on the constant $L$ in the definition  \eqref{eq:admgk} of $\cA^g_k$.
By assuming such a priori bound on the multiplicity,   the argument of Theorem \ref{thm:mult} still holds. In particular, multiplicity is conserved along the evolution.
Moreover, in the weaker metric setting given by the Wasserstein distance $\ove W_p$ on the spatial  supports, one can still prove existence of a
restricted GMM to $\cA^g_k$ by constraining the multiplicity.

\section*{Acknowledgments}
This work has been partially supported by the Austrian Science Fund (FWF) project F\,65 and by the BMBWF through the OeAD WTZ projects
CZ04/2019 
and
CZ01/2021, 
as well as their Czech counterpart  M\v{S}MT \v{C}R project 8J21AT001. 
K.\ Brazda acknowledges the support by the DFG-FWF international joint project FR 4083/3-1/I 4354 and the FWF project W\,1245. M.\ Kru\v{z}\'{\i}k is indebted to the E.~Schr\"{o}dinger Institute for Mathematics and Physics for its hospitality during his stay in Vienna in 2022. He also acknowledges support  by the GA\v{C}R-FWF project  21-06569K.
U. Stefanelli also acknowledges support from the FWF projects  I\,5149 and P\,32788.

\newcommand{\SortNoop}[1]{}

\bibliographystyle{plain}

\end{document}